\documentclass[12pt,reqno]{amsart}
\usepackage{amssymb}
\usepackage[curve,matrix,arrow]{xy}
\newtheorem{thm}{Theorem}[section]
\newtheorem{lemma}[thm]{Lemma}
\newtheorem{cor}[thm]{Corollary}
\newtheorem{prop}[thm]{Proposition}

\newtheorem{defi}[thm]{Definition}
\newtheorem{hyp}[thm]{Hypothesis}
\theoremstyle{definition}
\newtheorem{rem}[thm]{Remark}

\newtheorem{exm}[thm]{Example}

\def\CC{{\mathcal{C}}}

\def\CI{{\mathcal{I}}}
\def\CJ{{\mathcal{J}}}

\def\CN{{\mathcal{N}}}

\def\CS{{\mathcal{S}}}
\def\CW{{\mathcal{W}}}

\def\CV{{\mathcal{V}}}

\def\Fn{{\mathfrak{n}}}

\def\FN{{\mathfrak{N}}}

\def\FU{\mathfrak{u}}
\def\Fu{A}
\def\bZ{{\mathbb Z}}

\def\bP{{\mathbb P}}

\def\dim{\operatorname{dim}\nolimits}

\def\Ten{\operatorname{Ten}\nolimits}

\def\proj{\operatorname{(proj)}\nolimits}

\def\Det{\operatorname{Det}\nolimits}
\def\Ten{\operatorname{Ten}\nolimits}

\def\ovl#1{\overline{#1}}

\def\HHH{\operatorname{H}\nolimits}
\def\Hom{\operatorname{Hom}\nolimits}

\def\Dist{\operatorname{Dist}\nolimits}

\title{Endotrivial modules for finite group schemes}
\author[Jon F. Carlson]{Jon F. Carlson}
\thanks{Research of the first author partially supported by NSF grant
DMS-0654173}
\address{Department of Mathematics, University of Georgia,
Athens, Georgia 30602, USA}
\email{jfc@math.uga.edu}
\author[Daniel K. Nakano]{Daniel K. Nakano}
\thanks{Research of the second author partially supported by NSF grant
DMS-0654169}
\address{Department of Mathematics, University of Georgia,
Athens, Georgia 30602, USA}
\email{nakano@math.uga.edu}

\date\today

\begin{document}

\begin{abstract} It is well known that if $G$ is a finite group then the group 
of endotrivial modules is finitely generated.  In 
this paper we investigate endotrivial modules over arbitrary 
finite group schemes. Our results can be applied to computing the endotrivial group 
for several classes of infinitesimal group schemes which include 
the Frobenius kernels of parabolic subgroups, and their unipotent radicals 
(for reductive algebraic groups). 
For $G$ reductive, we also present a classification of simple, 
induced/Weyl and tilting modules ($G$-modules) which are endotrivial 
over the Frobenius kernel $G_{r}$ of $G$.  
\end{abstract}

\maketitle

\section{Introduction}

Let $A$ be a finite dimensional cocommutative Hopf algebra over a field $k$. 
An endotrivial module is an $A$-module $M$ with the property that 
$\Hom_k(M,M) \cong k \oplus P$ where $P$ is a 
projective $A$-module and the isomorphism is a map of $A$-modules.
Because $\Hom_k(M,M) \cong M \otimes_k M^*$
where $M^*$ is the $k$-dual of $M$, we have that $M \otimes -$ is a 
self equivalence of the stable category of all $A$-modules modulo 
projective modules. Thus the endotrivial modules determine a subgroup 
of the Picard group of self equivalences of the stable category. 
In addition, endotrivial modules form an interesting class of modules 
which in many cases is classifiable even though the category of all $A$-modules is 
wild, in general. 

In the case that $A= kG$ is the group algebra of a finite group $G$ 
with coefficients in a field of characteristic $p>0$, the endotrivial 
modules play a big part in the representation theory. 
When $G$ is a $p$-group, the first author and Th\'evenaz classified \cite{CaTh2,CaTh3} the endotrivial modules, 
building on the work of Dade, Alperin and 
others  \cite{A1,D1a, D1b}. This work was used by Bouc \cite{B}
in his classification of endopermutation
modules for $p$-groups. For other finite groups, the two authors together with
Mazza and Hemmer have computed the group of endotrivial modules for 
groups of Lie type \cite{CMN1} and 
for alternating and symmetric groups \cite{CMN2,CHM}.

The aim of this paper is to initiate an investigation 
of endotrivial modules for arbitrary finite 
group schemes. We are particularly interested in the case where the
group scheme is the Frobenius kernel of an algebraic group. 
Unfortunately, many of the methods used in the classification 
of endotrivial modules over finite groups do not work in the more general setting
or can be applied only after some adaptation. For arbitrary finite group 
schemes, we must rely on more geometric techniques. The methods 
work best in the case that 
the group scheme is unipotent or has a nontrivial unipotent radical. 
The unipotent group schemes are analogous
to $p$-groups in the realm of finite groups and here we can call upon the 
classification by Dade of the endotrivial modules over abelian $p$-groups 
\cite{D1a,D1b} (even though the Hopf algebra structures may be different). In
the case of a unipotent restricted $p$-Lie algebra (associated to a reductive 
algebra group), we obtain a complete classification of the endotrivial modules. 
In all of the examples that we have been able to work out, the group 
of endotrivial modules is finitely generated.  
Yet, for general finite group schemes, 
a proof of the finite generation of the group of endotrivial modules 
remains an open problem (see Section 9). 

We briefly describe the results of the paper. First, 
we set up the necessary notation that 
will be used throughout the paper in the next
section. In Section 3, we present general 
results on endotrivial modules for finite group 
schemes and show that if $G$ is a finite unipotent 
group scheme then for any fixed finite integer 
$n$ there is only a finite number of endotrivial 
$G$-modules of dimension $n$ (cf. Theorem 3.5). 
This indicates that for unipotent finite group schemes 
the problem of classifying endotrivial modules is not a wild problem. 
Furthermore in Section 4, for unipotent group schemes, a criterion 
is given via the connectedness of a subvariety 
of $\pi$-points, $\Pi(G)$, which is sufficient 
to prove that the group of endotrivial modules 
is isomorphic to ${\mathbb Z}$ (i.e., the endotrivial 
modules are given by syzygies of the trivial module). 
In that section, we have attempted to 
extract the exact hypotheses necessary to prove this result
in the hopes that the methods will have wider applications.

A semisimple algebraic group scheme $G$ has a
Borel subgroup $B$ with unipotent 
radical $U$ and, for $J$ a subset of the simple roots, 
parabolic subgroups $P_{J}$ and their 
associated unipotent radicals $U_{J}$. Let $G_{r}$, $B_{r}$, 
$(P_{J})_{r}$ and $(U_J)_r$ denote their $r^{th}$ infinitesimal 
Frobenius kernels. In Sections 5 and 6, 
we compute the endotrivial 
groups for $B_{1}$, $U_{1}$, $(P_{J})_{1}$ and $(U_{J})_{1}$ 
when $J\neq \varnothing$.  We show, using the methods of 
Section 4, that except in a couple of 
low rank cases the group of endotrivial modules for 
$U_1$ is generated by $\Omega(k)$, the syzygy of the 
trivial module. From this we can compute the endotrivial
group for $B_1$ and for $(P_{J})_{1}$. 
The lone remaining case for the first Frobenius kernel 
is the computation of the endotrivial group for $G_{1}$. 

For $G=SL_{2}$ we compute the endotrivial group 
for the first Frobenius kernel in Section 7, and prove that there 
are Weyl modules (which are not syzygies of the 
trivial module) which are endotrivial over $G_{r}$ when $r\geq 1$. 
This leads us to investigate the question of 
classifying simple, induced/Weyl, and tilting modules 
over $G$ which become endotrivial over $G_{r}$. 
In Section 8, we demonstrate that this cannot happen 
when the Lie rank of $G$ is greater than two. We suspect 
for large enough Lie rank that the group of endotrivial modules over $G_{r}$ is
isomorphic to ${\mathbb Z}$. In Section 9, we describe connections with the work of Balmer
on the endotrivial group and the Picard group of the projectivization of 
the cohomological spectrum. In the final section, we present some additional open problems that 
are inspired from the results of this paper. 

The authors would like to thank Ivo Dell'Ambrogio and Robert Varley for 
several useful discussions involving Picard groups and their connections 
with endotrivial modules.

\section{Notation and preliminaries}

In this section we set some notation and recall a few facts about 
representations of group schemes. Throughout the paper, let $k$ be a 
field of characteristic $p>0$. A finite group scheme $G$ over $k$ is a
group scheme over $k$ whose coordinate algebra $k[G]$ has finite $k$-dimension.  
The group algebra of $G$ is the dual of $k[G]$, the coordinate ring, 
and is denoted $kG$. The algebra $kG$ is a cocommutative 
Hopf algebra. In some cases 
$kG$ is isomorphic as an algebra to the group algebra of a finite group.
For example, if $G$ is a unipotent abelian group scheme, then $kG$
is isomorphic as an algebra to the group algebra of a finite abelian
$p$-group, though the coalgebra structures may be different.

By a $kG$-module we mean a finitely generated
$kG$-module. In particular, we need to define the 
rank variety for a $kG$-module in terms of $\pi$-points as in the work of Friedlander and 
Pevtsova \cite{FP}. Some of the notation that is introduced in this 
section is mostly necessary for the definitions of the rank variety.

A rational $G$-module is  both a module for the group algebra
$kG$ and comodule for the coordinate ring  $k[G]$.
For a field extension $K$ of $k$ let
$G_K$ be the base change of the $k$-group scheme
$G$ to the $K$-group scheme
$G_K = G \times_{\text{Spec}(k)} \text{Spec}(K)$. Note that
that the group algebra $KG_K$ of $G_K$ is the extension of
scalars,  $KG = K\otimes_k kG$. 

A $\pi$-point  for a finite group scheme $G$
is a flat map of $K$-algebras
$\alpha_K: K[t]/(t^p) \longrightarrow KG_K$, where $K$ is a  field
extension of $k$, such that the map factors through the group algebra
$KU_K \subset KG_K$ of some unipotent abelian subgroup
scheme $U_K \subset G_K$. A $\pi$-point is not assumed to be a 
map of Hopf algebras. 

To define the rank variety of $\pi$-points we must introduce an 
equivalence relation. We say that two $\pi$-points 
$\alpha_K, \beta_L$ are
equivalent if for every finitely generated $kG$-module $M$,
the $K[t]/(t^p)$-module $\alpha_K^*(M_K)$ is projective if and only if
the $L[t]/(t^p)$-module $\beta_L^*(M_L)$ is projective.
Here $\alpha_K^*(M_K)$ is the restriction of $M_K = K \otimes M$
to a $K[t]/(t^p)$-module along the map $\alpha_K$.

The set of equivalence classes
of $\pi$-points, which we denote by  $\Pi(G)$, 
has a scheme structure which is defined by the representation 
theory.  That is, a subset of $\Pi(G)$
is defined to be closed if it has the form $\Pi(G)_M$ where $M$
is a finite dimensional $kG$-module and $\Pi(G)_M$
is the subset of those equivalence
classes of $\pi$-points
$\alpha_K: K[t]/(t^p) \longrightarrow  KG$
such that $\alpha_K^*(M_K)$ is not a projective module.
With this definition, it can be proved that the scheme $\Pi(G)$ is homeomorphic
to the scheme $\text{Proj}(\HHH^*(G,k))$, the projective prime ideal spectrum
of $\HHH^*(G,k)$.

In the case that the field $k$ is algebraically closed we can get a 
simpler  formulation which is called the set of $p$-points \cite{FP1}.
A $p$-point is simply a $\pi$-point $\alpha: k[t]/(t^p) \longrightarrow kG$
which involves no field extension.  Two $p$-points are equivalent if 
they satisfy the same condition as above, i.e., for any $kG$-module
$M$, the restriction along one is projective if and only if the restriction 
along the other is projective. We also impose a topology on $\pi(G)$,
the set of equivalence classes of $p$-points, 
exactly as above.  With this topology, for $k$ algebraically closed 
$\pi(G)$ is homeomorphic
to the projectivized maximal ideal spectrum of $\HHH^*(G,k)$. 
This is the variety of closed points in $\text{Proj}(\HHH^*(G,k))$.
A $kG$-module $M$ is projective if and only if its restriction along
every $p$-point is projective.

Modules over the algebra $k[t]/(t^p)$ are classified by Jordan type,
i.e., by the Jordan canonical form of the matrix $A_t$ of the action
of the $t$ on the module. Because, $(A_t)^p = 0$, we know that all 
of the eigenvalues of $A_t$ are $0$ and no Jordan block has more than 
$p$ rows. We write that the Jordan type of a
$k[t]/(t^p)$-module $M$ is $a_p[p] + \cdots + a_1[1]$
if the matrix $A_t$ of $t$ on $M$ has $a_p$ blocks of size $p$,
$a_{p-1}$ blocks of size $p-1$, etc. Here,
\[
a_pp + a_{p-1}(p-1) + \dots + a_1 = n = \dim M.
\] 
The Jordan 
type is a partition of $n = \dim M$ having the form
$(p^{a_{p}},(p-1)^{a_{p-1}},\dots,1^{a_{1}})$, 
where for each $i$ there are $a_i$ of the entries 
in the partition with value $i$. 

The next result is well known, but we sketch a proof 
for the sake of completeness. 

\begin{lemma} \label{uni-dim-count}
Suppose that $G$ is a finite unipotent group scheme over $k$. Then
the socle of $kG$ has dimension one. Let $u$ be a generator for the
socle of $kG$. Suppose that $M$ is a $kG$-module and that $uM$ has 
dimension $r$. Then $M \cong (kG)^r \oplus N$ where $N$ has no 
projective submodules. 
\end{lemma}

\begin{proof}
The first statement is a consequence of the fact that $kG$ is a local
self-injective $k$-algebra with only a single isomorphism class of
irreducible modules. In particular, we note that every projective 
$kG$-module is free. The reader is referred to Chapter 8 in Part I 
of \cite{Jan} for more details. Let $u$ be a generator
for the socle and suppose that $M$ is a finitely generated  $kG$-module.
Let $um_1, um_2, \dots, um_r$ be a basis
of $uM$ for some elements $m_1, \dots, m_r \in M$. Let $P$ be a 
free $kG$-module of $kG$-rank $r$ having generators $e_1, \dots, e_r$.
Then the map $\psi: P \longrightarrow M$, given by $\psi(e_i) = 
m_i$ for $i = 1, \dots, r$ is an injection, because it is injective 
on the socle of $P$. But now, $P$ is also an injective module and
hence $\psi$ is left split. So, $M \cong P \oplus N$ for some $N$. 
We know that $N$ has no projective submodules because $uN = \{0\}$. 
\end{proof}


\section{Generalities on endotrivial modules}
In this section we consider some basic issues with endotrivial modules
over a unipotent group scheme $kG$. The results here will be important
in the remainder of the paper. Throughout the section we assume that the 
field $k$ is algebraically closed. First the definition.

\begin{defi}
Suppose that $G$ is any finite group scheme defined over $k$. A $kG$-module
is an endotrivial module provided that, as $kG$-modules,
\[
\Hom_k(M,M) \cong k \oplus P
\]
for some projective module $P$. 
\end{defi}

In other words, a $kG$-module $M$ is an endotrivial module if its 
$k$-endomorphism ring is isomorphic to the trivial module in the stable
category of $kG$-modules modulo projectives. For any $kG$-module $M$
there is a canonical isomorphism $\Hom_k(M,M) \cong M^* \otimes M$. Using
this we define the group $T(kG)$ of endotrivial $kG$-modules as follows. 
The objects in $T(kG)$ are the equivalence classes $[M]$ of endotrivial 
$kG$-modules $M$ where the equivalence relation is given by the rule that
$[M] = [N]$ if there are projective modules $P$ and $Q$ such that
$M \oplus P \cong N \oplus Q$. The operation on the group is given
by $[M] + [N] = [M \otimes N]$. Hence, the identity element is the class $[k]$
and the inverse of $[M]$ is the class $[M^*]$. 

The following theorem proved in \cite{CFP} is very useful for our 
analysis. Note that the Jordan type of a $kG$-module is completely
independent on the coalgebra structure. Consequently, the theorem
implies that the property of a $kG$-module being endotrivial is 
independent of the coalgebra structure. We emphasize that this result
depends on the field $k$ being algebraically closed. Otherwise, we must
rely on $\pi$-points rather than $p$-points.

\begin{thm} \label{pi-char-endotriv} 
Suppose that $G$ is a finite group scheme defined over $k$ and that 
$M$ is a finitely generated $kG$-module. Then $M$ is an endotrivial
module if and only if for any $p$-point $\alpha_k: k[t]/(t^p)
\longrightarrow kG$, the restriction $\alpha^*(M)$ has Jordan 
type either $[1] + n[p]$ or $[p-1] + n[p]$ for some integer $n$. 
\end{thm}

\begin{proof}
We note first that if $M$ is endotrivial then 
the Jordan type of any $p$-point is as stated. Hence, we need only
prove the converse. 
We know from Theorem 5.6 of \cite{CFP}, that $M$ is endotrivial if
$\alpha_K^*(M_K)$ has the prescribed Jordan type for all
$\pi$-points $\alpha_K:K[t]/(t^p) \longrightarrow KG_K$. We need 
only show that $M$ is endotrivial if the condition holds for $\pi$-points
defined over $K = k$. 

Suppose that for any $p$-point 
$\alpha:k[t]/(t^p) \longrightarrow kG$ the restricted module $\alpha^*(M)$
has Jordan type either $[1] + n[p]$ or $[p-1] + n[p]$ for some integer $n$.
Then the partition corresponding to this Jordan type is maximal in 
the dominance ordering of partitions of $\dim(M)$. Hence, such a $p$-point 
has maximal Jordan type. By \cite{FPS}, any $\pi$-point in the 
equivalence class of such an $\alpha$ also has maximal Jordan type. 
Any $\pi$-point $\alpha_K: K[t]/(t^p) \longrightarrow KG$,
must also have maximal Jordan type $[1] + n[p]$ or $[p-1] + n[p]$.
The reason is that, because $k$ is algebraically closed and $M$ is 
finite dimensional, the class of $\alpha_K$ must specialize to some 
closed point $[\beta]$ in $\Pi(G)$ and the Jordan type of $\alpha^*_K(M_K)$
must be larger than that of $\beta^*(M)$. As the closed points in 
$\Pi(G)$ are represented by $p$-points defined over $k$, the 
Jordan type of $\alpha^*_K(M_K)$ must be maximal. 
We conclude that $M$ has constant Jordan as defined in \cite{CFP},
and by the first statement of the theorem, $M$ is endotrivial. 
\end{proof}

An immediate application of the theorem is the following. 

\begin{lemma} \label{ele-endotriv}
Suppose that $G$ is a unipotent abelian group scheme. Then
any endotrivial $kG$-module is isomorphic to $\Omega^n(k) \oplus Q$ for 
some $n$ and some projective module $Q$.
\end{lemma}

\begin{proof}
The group algebra $kG$ is isomorphic to $kE$ where $E$ is an 
abelian $p$-group. This is an isomorphism as algebras, 
but not as Hopf 
algebras. By Theorem \ref{pi-char-endotriv}, 
a $kG$-module or a $kE$-module is endotrivial 
if and only if it has constant Jordan type $[1] + n[p]$ 
or $[p-1] + n[p]$ for some $n$. 
Consequently, any endotrivial $kG$-module is also an endotrivial
$kE$-module and hence is isomorphic to $\Omega^n(k) \oplus Q$ for some
some $n$ and some projective module $Q$, by the theorem of Dade 
\cite{D1a,D1b}. 
\end{proof}

Another consequence of the theorem is the following result on restrictions.

\begin{prop} \label{restr-endotriv}
Suppose that $G$ and $H$ are finite group schemes defined over $k$
and suppose that we have a flat map $\gamma: kH \rightarrow kG$. 
If $M$ is an endotrivial $kG$-module, then the pullback along 
$\gamma$ of $M$ is an endotrivial $kH$-module. 
\end{prop}

\begin{proof} Observe that regardless of the coalgebra
structure on $kH$, any $\pi$-point for $H$ becomes a $\pi$-point for 
$G$ when composed with the inclusion map $kH \rightarrow kG$. 
Consequently, $M|_{kH}$ has constant Jordan type exactly as required 
for an endotrivial module by Theorem \ref{pi-char-endotriv}.
\end{proof}

The next result follows the lines of a idea of Dade \cite{D2}, 
but for the sake of completeness we present a detailed proof.  The 
result is not as strong as what we can get in the case of a group 
algebra. That is, it does not follow automatically that the group 
of endotrivial modules is finitely generated. 

\begin{thm} \label{finite-per-dim}
Suppose that $G$ is a finite unipotent group scheme. For any positive integer
$n$, there is at most a finite number of endotrivial $kG$-modules of 
dimension $n$.
\end{thm}

\begin{proof}
We wish to consider the variety $\CV_n$ consisting of all 
representations of $kG$ of dimension $n$.
It can be described as follows. Suppose that $g_1, g_2, \dots, g_t$
is a set of generators of the algebra $kG$, so that 
every element of $kG$ can be written as a polynomial  in (noncommuting) 
variables $g_1, \dots, g_t$. Thus, $kG \cong k\langle g_1, 
\dots, g_t \rangle/\CJ$ for some ideal $\CJ$. 

The variety $\CV_n$ is defined to be the zero locus
in $k^{tn^2}$ determined by the ideal $\CJ$ of relations on the generators
$g_i$.  Specifically, the construction is the following.  
A representation of $kG$ of dimension $n$
is a homomorphism $\varphi: kG \longrightarrow M_n$ where $M_n$ is 
the algebra of $n \times n$ matrices over $k$. So for each 
$i$, the image of $g_i$ is a matrix $\varphi(g_i) = (a^i_{j,\ell})$.
In the general representation, we can consider elements 
$a^i_{j,\ell}$ to be variables or indeterminants. A relation among the 
generators of $kG$,  gives us a collection of relations among the 
variables. For example, if it were the case that $g_ig_j = 0$ 
(so that $g_ig_j \in \CJ$), then 
we would have the relations $f_{r,t} = \sum_{s=1}^n a^i_{r,s}a^j_{s,t} = 0$
for all $1 \leq r, t \leq n$. The elements $f_{r,t}$ are then elements
of an ideal $\CI \subseteq R = k[a^i_{j,\ell}]$, for  
$1 \leq i \leq t, 1 \leq j, \ell \leq n$. 
The variety $\CV_n$ can be taken to be the maximal
ideal spectrum of the ring $R/\CI$ or alternatively as the zero locus
of  $\CI$ in $k^{tn^2}$. For more information on the variety $\CV_n$ see
the paper of Dade \cite{D2}. 

Now let $u$ be a generator for the socle of $kG$. Recall that by 
Lemma \ref{uni-dim-count}, the dimension of $uM$ is the rank of the 
largest projective summand of the $kG$-module $M$.  
Suppose that $M$ is an endotrivial $kG$-module of dimension $n$. 
Then $n \equiv \pm 1\ (\text{mod $p$})$ and $M \otimes M^*
\cong k \oplus Q$ where $Q$ is projective module of dimension 
$p^ur$ where $p^u = \dim(kG)$ and $r$ is some positive integer. 
So $(\dim(M))^2 = 1 + \dim(Q)$. 

Now we choose a particular representation of $\rho: kG \longrightarrow
M_n$ of the dual module $M^*$. 
Let $\CS$ be the set of all subsets of $\FN = \{1, \dots, n\}$
having exactly $r$ elements. For any pair $S,T$ of elements of $\CS$
define $f_{S,T}: \CV_n \longrightarrow k$ by the rule that 
$f_{S,T}(\gamma) = \Det(M_{S,T}((\gamma \otimes \rho) (u)))$. Here 
$(\gamma \otimes \rho) (u)$ is the action matrix of $u$ on the module
$N \otimes M^*$ given by the tensor product of the 
representations $\gamma \otimes \rho$ where
$N$ is the module affording the representation $\gamma$. Also 
$M_{S,T}$ is the $r \times r$ submatrix whose rows are indexed by the 
elements of $S$ and whose columns are indexed by the elements of $T$. 
Finally,  $\Det$ is the determinant function. 

The important item to notice is that each $f_{S,T}$ is a 
polynomial map. Specifically, any entry of the matrix 
$(\gamma \otimes \rho) (v)$ is a polynomial in the entries of 
$\gamma(v)$ (remember $\rho$ and hence $\rho(v)$ are fixed) for  
generators $v$ of $kG$. Then $u$ is a polynomial in the generators
of $kG$. Also, $\Det$ is a polynomial function in the entries of the 
matrix. 

The result of this is that the zero locus of $f_{S,T}$ is a closed
set in $\CV_n$. Let $W_M$ be the closed set which is the intersection
of the zero loci of all $f_{S,T}$ for all $S, T \in \CS$. We claim
that if $\gamma \in \CV_n$ is not in $W_M$, then the module $N$ afforded
by $\gamma$ is isomorphic to $M$.  The reason for this is that if 
$\gamma \not\in W_M$, then the rank of $u$ on $N \otimes M^*$ is $r$. 
Therefore $N \otimes M^*$ has a direct summand isomorphic to $kG^r$.
The only way this could happen is if $N\otimes M^* \cong k \oplus kG^r$,
by a dimension argument. Hence $[N]=[M]$ as desired. 

The result we have proved says that if $M$ is an endotrivial $kG$-module
then the set of all representations isomorphic to 
$M$ is an open set in $\CV_n$
whose closure must then be a component of the variety $\CV_n$. Since
the variety has only a finite number of components, the theorem is 
proved. 
\end{proof}


\section{Unipotent finite group schemes with PBW bases.}
Throughout this section let $k$ be an algebraically closed field
of characteristic $p > 0$. In this section we wish to consider 
finite group schemes satisfying the following hypotheses.

\begin{hyp} \label{hyp1} 
We assume that $G$ is a finite unipotent group scheme having the property 
that the group algebra $A = kG$ has a set of generators $u_1, \dots, u_n$,
$n > 1$, with the properties that 
\begin{itemize}
\item[(a)] for each $i = 1, \dots, n$,  $u_i^p =0$,
\item[(b)] for each $i = 1, \dots, n-1$, the class of element $u_i$ is central in 
$kG/I_i$ where $I_i$ is the ideal generated by $u_{i+1}, \dots, u_n$,
\item[(c)] the  element $u_n$ is central in $A$, and 
\item[(d)] the dimension of $A$ is $p^n$.
\end{itemize}
\end{hyp}

An example of such a group algebra is the restricted enveloping 
algebra $u({\mathfrak n})$ of  a unipotent restricted $p$-Lie
algebra $\Fn$ such that $x^{[p]} = 0$  for every $x \in \Fn$. Such a 
Lie algebra has a basis $x_1, \dots, x_t$ such that $x_t$ is central
and corresponding to this is a basis for $A$ consisting of 
elements $u_i$ as in the hypothesis. 

It is likely that some of the conditions of Hypothesis \ref{hyp1} can be relaxed or 
altered and still imply the results that we prove later in the paper.  One of the 
things that we really need is that the algebra have a Poincar\'e-Birkhoff-Witt 
basis as in the following.

\begin{lemma} \label{conseq-hyp1}
Suppose that $A = kG$ has a collection of generators $u_1, \dots, u_n$
satisfying the conditions of Hypothesis \ref{hyp1}. Then we have the following.
\begin{itemize} 
\item[(a)] The algebra $A$ has a $k$-basis consisting of the set  
\[
S \ = \ \{u_1^{i_1} u_2^{i_2}  \cdots u_n^{i_n}  \ \vert \  0 \leq i_j < p,  \ \ j = 1 , \dots, n\}.
\]
\item[(b)] For any $s = 1, \dots, n$, the ideal $J_s = Au_{s} + \cdots + Au_n$ is a
two sided ideal and $J_s(u_s^{p-1} \cdots u_n^{p-1}) = 0$. That is, if $x$ is any 
element of $J_s$, then $xu_s^{p-1} \cdots u_n^{p-1} = 0$.
\item[(c)] Suppose that $v = a_1u_1 + \cdots + a_nu_n$ for $a_1, \dots, a_n \in k$ 
and $a_i \neq 0$ for some $i <n$. If $v^p = 0$, then $v^{p-1} \not\in Au_n$,
and the subalgebra $E = k\langle v, u_n \rangle$ generated by $v$ and $u_n$ 
is isomorphic to the group algebra of an elementary abelian $p$-group of rank 2.
\end{itemize} 
\end{lemma} 

\begin{proof}
The proof of the first statement is a standard argument. That is, we proceed by an 
induction on $n$, noting that the result is true if $n=2$ since the algebra $A$ 
in that case is commutative. A brief sketch of the argument goes as follows.
First notice that the dimension of $A/Au_n$ is precisely the number of Jordan
blocks of the matrix of the action of $u_n$ on $A$. Because none of the blocks
has more than $p$ rows and the dimension of $A$ is $p^n$, we conclude that
the dimension of $A/Au_n$ is at least $p^{n-1}$. Likewise the dimension of 
$A/(Au_{n-1} +Au_n)$ is the number of Jordan blocks of the matrix of the 
action of $u_{n-1}$ on $A/Au_n$. Again, we conclude that the dimension of 
$A/(Au_{n-1} + Au_n)$ is at least $p^{n-2}$. For this reason, we assume in 
the induction that follows, that the dimension of $A$ is at least $p^n$.

Now suppose that $n >2$, and by induction assume that the result is true whenever
there are fewer than $n$ generators. Thus by the assumption the result holds
for $A/Au_n$. Then any word in $u_1, \dots, u_n$ can be rewritten as a linear
combination of elements of $S$ and a linear combination of words which all
involve $u_n$ (which we recall is central). Each of these latter words can be 
written as the product $wu_n$ where $w$ is a word in $u_1, \dots, u_{n-1}, u_n$. Hence,
by induction, it can be written as a linear combination of elements of $S$ and
words which now involve $u_n^2$. Continuing this process we see that $S$ 
spans $A$. By a dimension argument $S$ must be a $k$-basis. 

The fact that $J_s$ is a two sided ideal is a consequence of parts $(b)$ and $(c)$ 
of Hypothesis \ref{hyp1}. To prove the second statement in $(b)$, we need only 
show that $u_j(u_s^{p-1} \cdots u_n^{p-1}) = 0$ for any $j \geq s$. For $j = s$ this
is obvious. For $j>s$ we need only note that $u_ju_s^{p-1} = u_s^{p-1}u_j + w$
for some $w$ in $J_{j+1}$ by the Hypothesis. Hence,  
\[
u_j(u_s^{p-1} \cdots u_n^{p-1}) \in J_{s+1}(u_{s+1}^{p-1} \cdots u_n^{p-1}) = 0
\] 
by a downward induction on $s$. This proves $(b)$.

Now assume the hypothesis of $(c)$. 
The subalgebra $E$ is commutative, and the two generators $v$ and $u_n$ have
the property that $v^p = 0 = u_n^p$. So there is a surjective homomorphism 
$\varphi: k[s,t]/(s^p,t^p) \longrightarrow E$ with $\varphi(s) = v$ and $\varphi(t) = u_n$.
To show that this is an injection we need only show that $v^{p-1}u_n^{p-1} \neq 0$,
since $s^{p-1}t^{p-1}$ is a $k$-basis for the minimal ideal of $k[s,t]/(s^p,t^p)$.
Now we can write $v = a_iu_i +a_{i+1}u_{i+1} + \dots + a_nu_n$ where $a_i \neq 0$.
That is, we assume that $i$ is the 
first index such that $a_i$ is not zero. Then 
we have that $v^{p-1} = a_i^{p-1}u_i^{p-1} + w$ 
where $w \in J_{i+1}$. Then 
\[ 
v^{p-1}(u_{i+1}^{p-1} \cdots u_n^{p-1}) = a_i^{p-1}u_i^{p-1}u_{i+1}^{p-1} \cdots u_n^{p-1} \neq 0
\]
by parts (a) and (b) of this lemma. 
\end{proof}

In addition we require the following hypothesis in some circumstances. For notation, 
let $\CV(A)$ be the set of points $\overline{a} = [a_1, \dots, a_n]$  in projective
$k$-space $\bP^{n-1}$ such that $v = \sum_{i=1}^n a_iu_i$ has the property
that $v^p = 0$.  

\begin{hyp} \label{hyp2}
Assume that $A=kG$ satisfies Hypothesis \ref{hyp1} with the 
given notation. We assume that  every equivalence class of $p$-points
has a representative $\alpha: k[t]/(t^p) \rightarrow A$, with the form
$\alpha(t) = \sum_{i=1}^n a_iu_i$ for elements $a_1, \dots, a_n$ such
that the projective point $\overline{a} = [a_1, \dots, a_n]$ 
is in  $\CV(A)$.  Moreover, we assume that the association of the class
of $\alpha$ to the point $\overline{a}$ induces a homeomorphism 
$\phi: \pi(G) \rightarrow \CV(A)$.

Let $\hat{\CV}(A)$ be the image of the projection 
\[
\CV(A) \setminus \{[0, \dots, 0,1]\}  \longrightarrow \bP^{n-2}
\] 
obtained by sending $[a_1, \dots, a_n]$ to $[a_1, \dots, a_{n-1}]$. 
We assume further that  $\hat{\CV}(A)$ is connected.   
\end{hyp}

Note in the hypothesis, that the existence of the induced homeomorphism $\phi$
requires that $\phi(\alpha) = \phi(\beta)$ if $\alpha$ and $\beta$ are equivalent
$p$-points. In addition, the map $\alpha:k[t]/(t^p) \rightarrow A$, defined by 
$\alpha(t) = u_n$ is a $p$-point and hence the element $[0, \dots, 0, 1] \in \bP^{n-1}$
is in $\CV(A)$. 

The result of that gets us started is the following. 

\begin{thm}  \label{uniform-syzygy}
Assume that $A=kG$ satisfies Hypotheses \ref{hyp1} and \ref{hyp2}. 
Suppose that $M$ is a finitely generated endotrivial $A$-module. There
exists a number $s$ with the property that if $\alpha:k[t]/(t^p) \rightarrow A$ 
is a $p$-point such that $\alpha(t) = a_1u_1 + \dots + a_nu_n$ with 
$a_i \neq 0$ for some $i < n$ and if   
$E = \langle v, u_n \rangle$ is the subalgebra of dimension
$p^2$ generated by $v= \alpha(t)$ and $u_n$, then $M|_E \cong \Omega^s(k) \oplus
\proj$. 
\end{thm}

\begin{proof}
By the hypothesis, 
any $p$-point $\alpha: k[t]/(t^p) \rightarrow A$, has an equivalent $p$-point 
such that $\alpha(t) = v_{\ovl{a}}  = \sum_{i=1}^{n}a_i u_i$
for some element $\ovl{a} = [a_1, \dots, a_n]$ of $\CV(A)$. 
If $a_i \neq 0$ for some $i <n$, let $E_{\ovl{a}} = \langle v_{\ovl{a}}, u_n\rangle$
be the subalgebra generated by $v_{\ovl{a}}$ and $u_n$. 
We emphasize that the class of the $p$-point does not depend on the
choice of the representative. That is, the $p$-point $\beta:k[t]/(t^p) \rightarrow kG$
with $\beta(t) = ca_1u_1 + \dots + ca_nu_n$ is equivalent to $\alpha$ for any
$c \neq 0$ in $k$. For our purposes it is important that if $a_i \neq 0$ for some $i <n$,
then the subalgebras $E_{\ovl{a}}$ and $E_{c\ovl{a}}$ are the same. 
Moreover, $E_{\ovl{a}}$ is the same as $E_{\ovl{a}^\prime}$, where
$\ovl{a}^\prime = [a_1, \dots, a_{n-1}] \in \hat{\CV}(A)$. Therefore, 
for any $\ovl{a} \in \hat{\CV}(A)$. We will use the notation 
$v_{\ovl{a}}$ and $E_{\ovl{a}}$ without further explanation.

Suppose that $\ovl{a} = [a_1, \dots, a_{n-1}] \in \hat{\CV}(A) \subseteq \bP^{n-2}$.
By Lemma \ref{conseq-hyp1}(c), 
there is a $p$-point $\alpha:k[t]/(t^p) \rightarrow A$ such that
$\alpha(t) = v_{\ovl{a}}  = \sum_{i=1}^{n-1}a_i u_i$.
Let $E_{\ovl{a}} = \langle v_{\ovl{a}}, u_n \rangle$. 
As an algebra $E_{\ovl{a}}$ is isomorphic to the group algebra of
an elementary abelian $p$-subgroup of rank 2 by Lemma \ref{conseq-hyp1}(c).

Let $M|_{E_{\ovl{a}}}$ denote the 
restriction of $M$ to a module over $E_{\ovl{a}}$. As such we have that 
$$
M|_{E_{\ovl{a}}} \cong \Omega^{m_{\ovl{a}}}(k) \oplus \proj 
$$
for some integer $m_{\ovl{a}}$. This is a consequence of Proposition 
\ref{restr-endotriv} and Lemma \ref{ele-endotriv}. Our purpose, then,
is to prove that $m_{\ovl{a}}$ is a constant, independent of $\ovl{a}$. Note that 
$m_{\ovl{a}} = m_{c\ovl{a}}$ for any nonzero $c \in k$. 

Next we observe that there exist numbers $b$ and $B$ such 
that $b \leq m_{\ovl{a}}  \leq B$. The reason is that for $m$ sufficiently large
we have
$$
\dim \Omega^m(k|_E) \ = \ \dim \Omega^{-m}(k|_E) \ >  \ \dim M
$$
for $E \cong E_{\ovl{a}}$ for any $\ovl{a}$. Now choose $b$ and $B$ so that for 
some $a$, $m_{\ovl{a}} = b$ and for some $\ovl{a}'$,  $m_{\ovl{a}'} = B$. Next we replace
$M$ by $\Omega^{-b}(M)$. Then we can assume that $b = 0$, and for every
$\ovl{a} \in \hat{\CV}(A) \subseteq \bP^{n-2}$, we have 
that $0 \leq m_{\ovl{a}} \leq B$. We are assuming 
further that for some $\ovl{a}$, $m_{\ovl{a}} = 0$, while for another, $m_{\ovl{a}} = B$. 

Let $C$ be any integer such that $0 < C < B$. Let 
$$
S_C = \{\ovl{a} \in \CV(\Fu) \setminus \{0\} \vert \ m_{\ovl{a}} > C\}.
$$ 
Our claim is that the set $S_C$ is closed in the Zariski topology of 
$\hat{\CV}(A)$.  The demonstration
that $S_C$ is closed is the main step and will occupy the next couple of 
paragraphs in the proof of the proposition.

Observe, that since $m_{\ovl{a}} = 0$ for some $\ovl{a}$, we must have that 
$\dim M \equiv 1 \ (\text{mod $p^2$})$. This implies that for all $\ovl{a}$,
$m_{\ovl{a}}$ is even. We also know that for $E \cong E_{\ovl{a}}$, any $\ovl{a}$, 
$\dim \Omega^{2s}(k|_E) = 1 + sp^2$ for any $s > 0$. So let 
$$
r = (\dim M - \dim \Omega^{2c}(k))/p^2
$$
where $c = C/2$ if $C$ is even and $c = (C-1)/2$ otherwise (so that in either
case $2c \leq C$ and $2c+2 > C$). The 
important thing to notice is that for any $a$, the statement
that $m_{\ovl{a}} < C$ means that the dimension of the projective part of 
$M_{E_{\ovl{a}}}$ is 
$$
\dim M - \dim \Omega^{m_{\ovl{a}}} (k) \ \geq \ p^2r.
$$
That is, if $m_{\ovl{a}} \leq C$, then $M_{E_{\ovl{a}}}$ 
has an $kE_{\ovl{a}}$-free summand of rank
at least $r$. In such a case, the rank of the matrix of the element
$w_{\ovl{a}} = v_{\ovl{a}}^{p-1}u_n^{p-1}$ (which generates the socle of $kE_{\ovl{a}}$) is 
at least $r$, by Lemma \ref{uni-dim-count}. On the other hand, by
a similar dimension argument, if $m_{\ovl{a}} > C$, then 
$M_{E_{\ovl{a}}}$ has no $E_{\ovl{a}}$-free summand of rank $r$, and the rank of
the matrix of $w_{\ovl{a}}$ is strictly less than $r$.  So we have that 
$\ovl{a} \in S_C$ if and only if the rank of $w_{\ovl{a}}$ is strictly less than
$r$. 

Now we follow the argument of Theorem \ref{finite-per-dim}.
Let $\CW$ be the set of all 
$a = (a_1, \dots, a_{n-1}) \in k^{n-1}$ such that
if $a \neq 0$, then $\ovl{a} = [a_1, \dots, a_{n-1}] \in \hat{\CV}(A)$. 
Let $d = \dim(M)$ and let 
$\CS$ be the set of all subsets of $\FN = \{1, \dots, d\}$
having exactly $r$ elements. For any pair $S,T$ of elements of $\CS$
let $f_{S,T}: \CW \longrightarrow k$ be given by
$$
f_{S,T}(a) = \Det(M_{S,T}(w_a))
$$
for $w_{a} = (a_1u_1 + \dots + a_{n-1}u_{n-1})^{p-1}u_n^{p-1}$ and 
$a = (a_1, \dots, a_{n-1})$. 
Here, $M_{S,T}(w_{a})$ is the $r \times r$ submatrix of the matrix of $w_a$ 
on the module $M$ having rows indexed by $S$ and columns indexed 
by $T$. It can be checked that the functions $f_{S,T}$ are homogeneous polynomial 
maps in $a_1, \dots, a_{n-1}$. 
Hence, we have a polynomial map
$$
f = \displaystyle\Pi_{S,T\in \CS} \ \ f_{S,T}: \CW
\longrightarrow k^{{\binom{d}{r}}^2}
$$
with the property that $f(a) = 0$ if and only if $\ovl{a} \in S_C$. 
where $\ovl{a} = [a_1, \dots, a_{n-1}]$ is the corresponding element of $\hat{\CV}(A)$.
Because the zero locus of $f$ in $\CW$ is a homogeneous closed set, we have that  
$S_C$ is closed in $\hat{\CV}(A)$. 

Now we claim that for any such number $C$, $S_C$ is also open
in $\hat{\CV}(A)$. To see this
fact, we replace $M$ by it dual $M^*$. 
Recall, that for any $t$, $(\Omega^t(k))^* \cong \Omega^{-t}(k)$.
For $M^*$ the values of $m_{\ovl{a}}$ 
are all between $-B$ and 0. Hence the next thing is to replace
$M^*$ by $\Omega^B(M^*)$ so that the values of $m_{\ovl{a}}$ for this module
are all between 0 and $B$. But note that we have reversed the ordering.
That is, for any $\ovl{a}$, we have that 
$$
M|_{E_{\ovl{a}}} \cong \Omega^{m_{\ovl{a}}}(k) \oplus \proj ,
$$
and so 
$$
(\Omega^B(M^*))|_{E_{\ovl{a}}} \cong  \Omega^{B-m_{\ovl{a}}}(k) \oplus \proj. 
$$
So that $S_C$ on $M$ becomes the complement of $S_{B-C}$ in $\hat{\CV}(A)$.
Hence, $S_C$ is both open and closed. Therefore, by the 
connectedness of $\hat{\CV}(A)$ (Hypothesis \ref{hyp2})
we have that $B = 0$, and the proposition
follows. 
\end{proof}

For the main result of this section we need  a third hypothesis. 

\begin{hyp} \label{hyp3}
Suppose that $G$ is a finite group scheme defined over $k$ and that the 
algebra $A = kG$ satisfies Hypotheses \ref{hyp1} and \ref{hyp2}. 
We assume further that the ideal $Au_n  = kGu_n$ generated by $u_n$ is a 
Hopf ideal so that $A/Au_n$ is a Hopf algebra and the group algebra
of a finite group scheme.  Moreover, we suppose that for any $p$-point 
$\alpha_k: k[t]/(t^p) \rightarrow A/Au_n$ one of the following 
two things happens.
\begin{itemize}
\item[(a)] There exists a $p$-point $\beta: k[t]/(t^p) \rightarrow A$ such that
the composition 
\[
k[t]/(t^p) \longrightarrow kG \longrightarrow A/Au_n 
\] 
is a $p$-point equivalent to $\alpha$.
\item[(b)] There exist an element $x \in A$ such that $x 
+ \langle u_n \rangle = \alpha(t) \in A/Au_n $, and 
we have that $x^p = yu_n$, where $y$ is a unit in $A$.
\end{itemize}
\end{hyp}

This last hypothesis seems somewhat unusual, but it is exactly what is needed. 
It allows us to prove the next theorem which is the main result of this section. 

\begin{thm} \label{char-endotriv}
Suppose that $A$ is an algebra which satisfies Hypotheses \ref{hyp1},
\ref{hyp2} and \ref{hyp3}.
Then the group of endotrivial modules for $A$ has the form
$T(A) \cong \bZ$ and is generated by the class of $\Omega(k)$.
\end{thm}

\begin{proof}
Let $M$ be an endotrivial $A$-module, and suppose that the 
restriction $M|_E \cong \Omega^m(k_E) \oplus \proj$ 
where $E$ is the subalgebra 
generated by $u_{n-1}$ and $u_n$. Our objective is to prove that 
$M \cong \Omega^m(k) \oplus \proj$. This is equivalent to proving that
$\Omega^{-m}(M) \cong k$, since $\Omega^{-m}(M)$ is indecomposable.
Hence, for the proof we replace $M$ by $\Omega^{-m}(M)$ and we 
can assume that $M|_E \cong k|_E \oplus \proj$. 

Our first objective is to
demonstrate that the module $\hat{M} = u_n^{p-1}M$ is a free 
module over $\hat{A} = A/Au_n$. For this purpose we use the method
of $p$-points. Specifically, we show that if $\alpha:k[t]/(t^p)  \rightarrow 
\hat{A}$ is a $p$-point, then the restriction $\alpha^*(\hat{M})$
is free module over $k[t]/(t^p)$. This is sufficient to prove that $\hat{M}$ is a 
projective $\hat{A}$-module.

Suppose that $\alpha$ is a $p$-point as above. We must consider the two
situations of Hypothesis \ref{hyp3} separately. First suppose that we are in 
situation (a). In this case $\alpha$ is equivalent to $q\circ \beta$ where 
$\beta:k[t]/(t^p) \rightarrow A$ is a $p$-point and 
$q: A \rightarrow A/Au_n$ is the quotient map.  By Hypothesis 
\ref{hyp2} we may assume that $v = \beta(t) = a_1u_1 + \dots + a_nu_n$ where 
$a_i \neq 0$ for some $i < n$. Let $E = \langle v, u_n \rangle$. By Theorem
\ref{uniform-syzygy} we have that $M \vert_E \cong k \oplus P$  where $P$ is
a free $E$-module. As a consequence, $u_n^{p-1}M = u_n^{p-1}P$ is a free module 
on restriction to $k[t]/(t^p)$ along the map $\beta$. But then by the equivalences
of $p$-points $\alpha^*(\hat{M})$ is also a free module. 

Now consider the situation (b) in Hypothesis \ref{hyp3}. 
Let $x \in A$ be as in the statement of (b), and let $H = \langle x \rangle$ be 
the subalgebra generated by $x$. Notice that $x^{p^2} = y^pu_n^p = 0$, while
$x^{p^2-1} = y^{p-1}x^{p-1}u_n^{p-1} \neq 0$ since $x^{p-1}$ is not in the ideal 
generated by $u_n$ (see Hypothesis \ref{hyp1}). Consequently, the subalgebra
$H$ is isomorphic to $k[t]/(t^{p^2})$, which is the group algebra of a cyclic group
of order $p^2$ and hence is a finite group scheme. For any $p$-point 
$\gamma: k[t]/(t^p) \rightarrow H$ we must have that $\gamma(t) = zx^p = zyu_n$ 
where $z \in U$ is a unit. That is, there is only one equivalence class of $p$-points. 
However, we know that the restriction of $M$ to the subalgebra generated by 
$u_n$ is  a trivial module plus a projective module. So the Jordan type of 
$\gamma(t)$ on $M$ must be the same as that of $u_n$ on $M$ which is 
$s[p] + [1]$ for some $s$. Hence, by Theorem \ref{pi-char-endotriv}, the restriction
to $H$ of $M$ is an endotrivial module. Because the dimension of $M$ is 
congruent to 1 modulo $p^2$ we know that $M \cong k \oplus P$ for some projective
module $P$. That is, we recall that $\Omega^2(k\vert_H) \cong k$ and that the 
dimension of $\Omega(k\vert_H)$ is congruent to -1 modulo $p^2$. From this 
we get that $\alpha^*(\hat{M})$ is projective as in the previous case. As 
noted above, we have now shown that $\hat{M}$ is a free $\hat{A}$-module.

Now note that the socle of $\hat{A}$ is generated by the element
$w = u_1^{p-1} \dots u_{n-1}^{p-1}$. Hence the dimension of 
$\hat{M}$ is $p^{n-1}\dim(w\hat{M})$. Now $y = u_n^{p-1}w$ is a 
generator for the socle of $A$ and $w\hat{M} = yM$. Therefore, 
the module $M$ has a free submodule $Q$ of dimension $p^n \dim(yM)$.
Because $p \dim(\hat{M}) = \dim(M) -1$ we must have that 
$M \cong k \oplus Q$ by Lemma \ref{uni-dim-count}. 
\end{proof}

\begin{rem}
The importance of Hypothesis \ref{hyp2} is highlighted in the following
example. Suppose that $U_1$ is the first infinitesimal Frobenius kernel of 
the unipotent subgroup of the algebraic group $G = SL_3$. That is, 
$kU_1$ is the restricted enveloping restricted $p$-Lie algebra of 
upper triangular $3 \times 3$ matrices. Then $kU_1$ satisfies 
Hypothesis \ref{hyp2} except in the case that $p = 2$. It is generated
by elements $u_1, u_2$ and $u_3$ which satisfy the equations
$$
u_1^p = 0,  \
u_2^p = 0,  \
u_3^p = 0,  \quad 
\text{and}  \quad u_3 = u_1u_2 - u_2u_1. 
$$
In the case that $p >2$, we have that $(au_1 + bu_2)^p = 0$, hence 
Hypothesis \ref{hyp2} holds. However, if $p=2$, this is not the case. 
Indeed, if $p=2$ then the last equation can be replaced by 
$(u_1u_2)^2 = (u_2u_1)^2$. These are the equations of the group 
algebra for a dihedral group. Consequently, $kU_1$ is isomorphic
to the mod-2 group algebra of a dihedral group of order 8, at least
as an algebra, though perhaps not as a Hopf algebra. Again we
recall from \cite{CFP} that the group of  endotrivial module does 
not depend on the coalgebra structure. Therefore, $T(kU_1)
\cong \bZ \oplus \bZ$ (see \cite{CT1}).
\end{rem} 


\section{Infinitesimal unipotent subgroups of algebraic groups}

The goal for the remainder of the paper is to investigate 
the properties of endotrivial modules for infinitesimal 
group schemes obtained by taking the Frobenius kernels 
of closed subgroup schemes of semisimple algebraic group schemes. 
In this section, we investigate some situations in which the 
unipotent subschemes satisfy the conditions of the last section.

Suppose that  $G$ is  a semisimple simply connected 
algebraic group defined and split over the finite 
field ${\mathbb F}_p$ with $p$ elements for a
prime $p$. Here  $k$ is the algebraic closure of ${\mathbb F}_p$.  Let
$\Phi$ be a root system associated to $G$ with respect to a maximal
split torus $T$. Let $\Phi^{+}$ (resp. $\Phi^{-}$) be the set of positive
(resp. negative) roots and $\Delta$ be a base consisting of simple roots. 
Let $B$ be a Borel subgroup containing $T$ corresponding 
to the negative roots and let $U$ denote 
the unipotent radical of $B$. 

If $H$ is an affine algebraic group scheme 
over $k$ and let $H_{r}=\text{ker }F^{r}$. 
Here $F:H\rightarrow H^{(1)}$ is the Frobenius map
and $F^{r}$ is the $r^{th}$  iteration of the Frobenius map.
We note that there is a categorical equivalence between restricted
$\text{Lie}(H)$-modules and $H_{1}$-modules.
For each value of $r$, the group algebra $kH_r$ is the distribution 
algebra $\Dist(H_r)$ (see \cite{Jan}). In general, for the rest of this 
paper, we use $\Dist(H_{r})$ to denote the group algebra of $H_r$. 

\begin{thm} \label{unipotent1}
Let $U$ be the unipotent radical of a Borel subgroup. For any $r \geq 1$, the algebra 
$\Dist(U_r)$ has a collection of elements $u_1, \dots, u_n$
satisfying the conditions of Hypothesis \ref{hyp1}. Here, $n = r\vert \Phi^+ \vert$.
\end{thm}

\begin{proof}
For any root $\alpha \in \Phi^+$, the distribution algebra of the 
infinitesimal root subgroup $(U_{\alpha})_r$ is a divided powers
algebra generated by elements 
\[
v_{\alpha,i} = \frac{x^{p^{i-1}}_{\alpha}}{p^{i-1}} \qquad \text{for} \quad i \ = \ 1, \dots, r.
\]
By convention we will set $v_{\alpha}=v_{\alpha,1}$. Hence, 
we let $n = r\vert \Psi^+ \vert$ and let $u_1, \dots, u_n$ be the 
elements $v_{\alpha,i}$ for $\alpha \in \Psi^+$ 
and $i = 1, \dots, r$, ordered in such
a way that for $u_a = v_{\alpha,i}$ and $u_b = v_{\beta,j}$, we have that 
$a < b$ if the height of the root $\alpha$ is 
less than the height of the root $\beta$.
We know that for any positive roots $\alpha$ and $\beta$, the commutator
$v_{\alpha,i}v_{\beta,j}- v_{\beta,j}v_{\alpha,i}$ 
is contained in the distribution
algebra of the root subgroup $U_{\alpha+\beta}$ or is zero if $\alpha+\beta$ is
not a root. In particular, that Condition (a) of 
Hypothesis \ref{hyp1} is satisfied
is a standard calculation. Conditions 
(b) and (c) of Hypothesis \ref{hyp1} hold because the elements of collection 
$u_1, \dots, u_n$ are ordered by increasing 
height of the roots.  In particular,  
$u_n = v_{\alpha,j}$ where  $\alpha$ has maximal height. Finally, Condition
(d) is well known to be satisfied for $\Dist(U_r)$. 
\end{proof}

Our next goal is to determine when $T(U_1) \cong \bZ$. 
Theorem \ref{char-endotriv} identifies the 
problem very clearly. We need to show that, except 
in the cases mentioned, Hypothesis \ref{hyp2} 
is satisfied. Throughout the proof, we assume the notation 
of the proof of Theorem \ref{unipotent1}.
In particular, the elements $u_1, \dots, u_n$ have the form $v_{\alpha,j}$ 
with the given ordering. 

The first thing to notice is that the distribution 
algebra $\Dist(U_1)$ is isomorphic 
to the restricted enveloping algebra of the 
restricted $p$-Lie algebra $\FU$ of $U$. 
The restricted enveloping algebra is defined to be the quotient $\Ten(\FU)/J$
where $\Ten(\FU)$ is the free tensor algebra 
\[
\Ten(\FU) = k \ \ \oplus \ \ \FU \ \ \oplus \ \ \FU \otimes \FU \ \  \oplus \ \ 
\FU \otimes \FU \otimes \FU \ \  \oplus \ \ \ldots 
\] 
and the ideal $J$ is generated by all elements of the form
\[
u \otimes v - v\otimes u  - [u,v] \quad 
\text{and} \quad u^{\otimes p} - u^{[p]}
\]
for all $u,v \in \FU$. Here $u \rightarrow u^{[p]}$ is the $p$-power operation 
on the restricted Lie algebra $\FU$. Now by \cite{Jan1}, we know that
spectrum of $\text{H}^*(u(\FU),k)$ is the restricted nullcone
\[
\CN_1(\FU) \ = \ \{u \in \FU \ \vert \ u^{[p]} = 0\}
\] 
where we are thinking of $\FU$ as a subspace of $\Dist(U_1)$. That is, every
class of $p$-points is represented by a $p$-point having the form
$\alpha: k[t]/(t^p) \longrightarrow \Dist(U_1)$, with $\alpha(t) \in \FU$.
As a result, we may assume that 
$\CV(A) = \CN_1(\FU) \subset \FU$. The elements 
$u_1, \dots, u_n$ can be taken to be a $k$-basis of $\FU$. 

Notice that if the prime $p$ is larger than the 
Coxeter number, then $\hat{\CV}(A) = k^{n-2}$ 
which is connected. Hence, we need only worry 
about small prime characteristics. 

\begin{prop}\label{patleast3}
If $p \geq 3$, then $A = \Dist(U_1)$ satisfies Hypothesis \ref{hyp2}.
\end{prop}

\begin{proof}
Suppose that $\ovl{a} = [a_1, \dots, a_n] \in \CV(A)$ with $a_i \neq 0$ for some
$i < n$.  The point  is that if $v = \sum a_iu_i$ and $p \geq 3$, then 
for any $a,b \in k$, not both zero, we have that $(av +bu_{n-1})^p = 0$. That is, 
the choice of the ordering on the elements $u_i$ insures that $u_{n-1}$
has the form $v_{\alpha}$ for $\alpha \in \Phi^+$ having next to greatest
height. Because, $v^p = 0$ and $p \geq 3$, we have that $(av+bu_{n-1})^p$ involves 
root subalgebras $(U_{\beta})_1$, where $\beta$ has height at least two 
greater than $\alpha$. As there is no such 
$\beta$, we conclude that $(av+bu_{n-1})^p = 0$, as asserted. Hence $\hat{\CV}(A)$
is connected.
\end{proof}

\begin{prop} \label{Gnotsimple}
If the algebraic group $G$ is semisimple but not simple, then 
$\Fu = \Dist(U_1)$ satisfies Hypothesis \ref{hyp2}.
\end{prop}

\begin{proof}
We are assuming that $G$ is semisimple. If $G$ is not simple, then 
$G=G_{1}\times G_2 \times \dots \times G_t$ and 
$A$ is a sum of algebras $A =  A_1 \oplus  A_2 \oplus \dots A_{t}$. 
Each of the
subalgebras has a nontrivial center. So there exist
some $u_i$ for $i < n$ that is also central. In particular, we can assume 
that $u_{n-1}$ is also central. 
Consequently, if $v^p = 0$, 
then so also is $(av+bu_{n-1})^p = 0$ for any
$(a,b) \in k^2$ with $a$ and $b$ not both zero. Thus, the points in 
$\hat{\CV}(A)$ corresponding to $v$ and $u_{n-1}$ 
are in the same connected component
of $\hat{\CV}(A)$. Hence, there is 
only one connected component.
\end{proof}

For the purposes of proving Theorem \ref{unipotent2}, we may assume that 
$r=1$, $p=2$ and that $G$ is a simple algebraic group. The case that 
$G$ has type $A_2$ was considered in the last section. Hypothesis 
\ref{hyp2} is vacuously satisfied (because $n=1$) if $G$ has type $A_1$.
The remaining cases of groups of Lie rank 2 are settled in the following. 

\begin{prop} \label{Lierank2}
Assume that $p = 2$. 
\begin{enumerate}
\item If $G$ has type $B_2$, then $\CV(A)$ has exactly two nonempty connected
components and Hypothesis \ref{hyp2} is not satisfied by $A = \Dist(U_1)$.
\item If $G$ has type $G_2$, then $A = \Dist(U_1)$ satisfies Hypothesis \ref{hyp2}.
\end{enumerate}
\end{prop}

\begin{proof}
In the first case, we may assume that $A = \Dist(U_1)$ is generated by elements
$v_{\alpha}, v_{\beta}, v_{\alpha+\beta}$ and  $v_{2\alpha+\beta}$, where $\alpha$
and $\beta$ are the simple roots with $\beta$ long. Note that $v_{2\alpha+\beta}$
is central (corresponding to $u_n$ in the previous notation). We may assume that
$v_{\alpha} v_{\beta} + v_{\beta} v_{\alpha} = v_{\alpha + \beta}$ and that 
$v_{\alpha} v_{\alpha+\beta} + v_{\alpha+\beta} v_{\alpha} = v_{2\alpha + \beta}$.
Then expanding the equation 
\[
(av_{\alpha} +b v_{\beta} + cv_{\alpha+\beta})^2 \ = \ 0 
\]
we get that $ab =0$ and $ac =0$. Hence, $\hat{\CV}(A)$ has two components 
which intersect in $0$. 

The argument for $G$ of type $G_2$ is similar but with a different result. The 
distribution algebra is generated by,  
\[
v_{\alpha}, \ v_{\beta},  \ v_{\alpha+\beta}, \ v_{2\alpha+\beta}, \ v_{3\alpha+\beta},
\ \  \text{and} \ \ v_{3\alpha+2\beta}
\]
where $\alpha$ and $\beta$ are the simple roots with $\beta$ long. This time,
$v_{3\alpha+2\beta}$ is central. Expanding the equation
\[
(av_{\alpha}+b v_{\beta} + cv_{\alpha+\beta} + dv_{2\alpha+\beta}+ 
ev_{3\alpha+\beta})^2 \ = \ 0
\]
we get that $ab = ac = ad = cd+be=0$. If $b= c = d = 0$, then there is no
restriction on $e$. On the other hand, if $a=0$ then the component defined
by $cd+be = 0$ is connected. 
Therefore, every $\ovl{a}$ in $\hat{\CV}(A)$ is in the connected 
component of $[0,0,0,0,1]$ which corresponds to $v_{3\alpha+\beta}$. 
\end{proof}

\begin{rem} \label{B2rem}
It can be shown, as in the case of $A_2$ in the last section, that if $G$ has type
$B_2$ and $p=2$, then $T(U_1) \cong \bZ\oplus \bZ$. 
\end{rem}

The following theorem describes the exact conditions when Hypothesis 
\ref{hyp2} is satisfied. 

\begin{thm} \label{unipotent2}
Let $U$ be the unipotent radical of a Borel subgroup of the semisimple 
algebraic group $G$. Then $A = \Dist(U_1)$ satisfies Hypothesis 
\ref{hyp2} except in the cases that  
$p =2$, and $G$ has type either $A_2$ or $B_2$. 
\end{thm}

\begin{proof} 
From Propositions \ref{patleast3}, \ref{Gnotsimple}, \ref{Lierank2}
and Remark \ref{B2rem}
stated above, we have reduced the proof of 
this theorem to dealing with the case that $p =2$, $r=1$ 
and $G$ is simple having Lie rank $\ell \geq 3$. 
The distribution algebra $A = \Dist(U_1)$ has generators $u_1, \dots, u_n$ 
where for each $i$, $u_i = x_{\alpha}$ for $\alpha$ a positive root. That is, the 
elements $u_i$ can be indexed by the elements of the set $\Phi^+$ of positive
roots. Recall that the ordering on $u_1, \dots, u_n$ respects the height of the
corresponding roots. Thus $u_n$ corresponds to a root of maximal height. 

In the arguments that follow we require some information about the three roots
of greatest height in $\Phi^+$. This information is given in the following table. 
For each type of simple root systems and for each of the three roots of 
maximal height, the table gives the  coefficients of the root expressed as 
a sum of simple roots. The ordering on the simple roots follows that of Bourbaki 
as given in \cite{Hum}. This table was compiled from information given in 
\cite{BMP} \cite{Hum}. 

\begin{center}
\begin{tabular}{|c|c|c|c|}\hline
Type & $\beta_{1}$ & $\beta_{2}$ & $\beta_{3}$ \\ \hline
$A_l$ & $(1,1,\dots,1,0)$ & $(0,1,1,\dots,1)$ 
& $(1,1,\dots,1)$ \\ \hline

$B_l$ & $(0,1,2, \dots, 2)$ & $(1,1,2, \dots, 2)$ &  $(1,2,2, \dots, 2)$ \\\hline

$C_l$ & $(1,1,2,\dots,2,1)$  & $(1,2,\dots,2,1)$ & $(2,2,\dots,2,1)$ \\ \hline

$D_l$& $(0,1,2,\dots, 2 ,1,1) $ & $(1,1,2,\dots, 2 ,1,1) $ & $(1,2,\dots, 2 ,1,1) $ \\ \hline 

$E_6$ & $(1,1,2,2,2,1)$  & $(1,1,2,3,2,1)$  &  $(1,2,2,3,2,1)$  \\  \hline 

$E_7$ & $(1,2,2,4,3,2,1)$ & $(1,2,3,4,3,2,1)$ & $(2,2,3,4,3,2,1)$ \\ \hline

$E_8$ & $(1,2,4,6,5,4,3,2)$ & $(1,3,4,6,5,4,3,2)$ & $(2,3,4,6,5,4,3,2)$ \\ \hline

$F_4$ & $(1,2,4,2)  $ & $(1,3,4,2) $ & $(2,3,4,2) $ \\ \hline
\end{tabular}
\end{center}

The first case that we consider is the one in which the root system of $G$ has 
two roots with (the same) next to maximal height. This is the case, for example,
if $G$ has type $A_n$. In this case, $u_{n-2}$ and $u_{n-1}$ correspond to 
roots $\beta_1$ and $\beta_2$ having the same height (see the table). Because the 
Lie rank of $G$ is larger than 2, we must have that any element of the form
$v = au_{n-2}+bu_{n-1}$ has $v^2 = 0$ and hence  $[0, \dots, a,b] \in \hat{\CV}(A)$. 
As a consequence, all such
elements are in the same connected component (call it $\CC$) 
of $\hat{\CV}(A)$. 
Suppose that $\ovl{a} = [a_1, \dots, a_{n-1}]$ is in $\hat{\CV}(A)$. Let 
$v = a_1u_1 + \dots + a_{n-1}u_{n-1}$. Then the commutators
$[v, u_{n-2}]= vu_{n-2}-u_{n-2}v$ and $[v,u_{n-1}]$ are 
both elements of the space generated by
$u_n$, because of the rank considerations. Consequently, there is some 
combination $w = cu_{n-2}+du_{n-1}$ such that $[v,w] = 0$.  Hence, we see that
$\ovl{a}$ is in the connected component $\CC$ that contains 
the point corresponding to $w$. 

The other possible case is that $u_{n-2}$ and $u_{n-1}$ correspond to roots of
different height. That is, we are assuming that there is only one root of next to 
maximal height. Let $\Delta = \{\alpha_1, \dots, \alpha_{\ell}\}$ be the simple
roots. There is no loss of generality in assuming here that $u_i = x_{\alpha_i}$
for $i = 1, \dots, \ell$. Let $\beta_1, \beta_2$ and $\beta_3$ be the roots 
corresponding to the elements $u_{n-2}, u_{n-1}$ and $u_n$, respectively (see the table). 
Our assumption on heights requires that there be some
simple root $\alpha_i$ and another simple root $\alpha_j$ such that 
$\beta_2 = \beta_1+\alpha_i$ and $\beta_3 = \beta_2 + \alpha_j$.
In these circumstances, we know that $\alpha_i + \alpha_j$ is a root of height 2.

As in the previous case, because $\ell \geq 3$ we know that $\beta_1$ must have
height at least 2, and so $u_{n-2}$ and $u_{n-1}$ are in the same connected 
component $\CC$ of $\hat{\CV}(A)$. That is, the commutator $[u_{n-2}, u_{n-1}] = 0$.
Moreover any linear combination of $u_{n-2}$ and $u_{n-1}$ is also in $\CC$.
Choose any element $v = a_1u_1 + \dots + 
a_{n-1}u_{n-1}$ with $v^2 = 0$.  We must have that 
$a_ia_j = 0$ ($i$ and $j$ as above), as otherwise $v^2$ would 
have a nonzero coefficient on the basis
element $u_s$ corresponding to the root $\alpha_i+\alpha_j$. So either 
$a_i=0$ or $a_j =0$. In the second case ($a_j =0$), the commutator $[v,u_{n-1}] =0$. 

So consider the first case, that $a_i = 0$. then we must have that $[v,u_{n-2}] = 
ca_su_n$ for some constant $c$ where $u_s$ corresponds to the root $\alpha_i+\alpha_j$.
At the same time $[v,u_{n-1}] = da_ju_n$. for some constant $d$. 
The numbers $c$ and $d$ only depend on the choice of the basis $u_1, \dots, u_n$ 
and not on the choice of $v$. So in any event, some linear combination $w = au_{n-2} 
+bu_{n-1}$ has the property that $[v,w] = 0$. Hence, the point of $\hat{\CV}(A)$ 
corresponding to $v$ is in the same connected 
component as that of $w$ and that component is $\CC$.
\end{proof}

We can now determine the group of endotrivial modules for $U_r$, 
in the case that $r = 1$.

\begin{thm} \label{unipotent3}
Let $U$ be the unipotent radical of a Borel subgroup of the semisimple 
algebraic group $G$. Then $T(U_1) \cong \bZ$ and is generated by 
the class or $\Omega(k)$ except in the cases that $\Phi=A_{1}$ or that 
$p =2$ and $\Phi=A_2$ or $B_2$. 
\end{thm}

\begin{proof}
Let $\FU$ be the restricted $p$-Lie algebra of $U$. We have shown that 
$\Dist(U_1)$ satisfies Hypothesis \ref{hyp1} and in the proof we 
see that the elements $u_1, \dots, u_n$ satisfying the hypothesis can be 
taken to be a basis of $\FU \subseteq \Dist(U_1)$. 

All of the above applies to the group scheme $\hat{U}_1$ whose distribution 
algebra is $\Dist(U_1)/\Dist(U_1)u_n$. This is the restricted enveloping algebra
of the restricted $p$-Lie algebra $\hat{\FU}/\langle u_n \rangle$. 
Here $\langle u_n \rangle$ denotes the Lie subalgebra generated by
$u_n$ and having dimension one. Any equivalence
class of  $p$-points for $\Dist(\hat{U}_1)$ is represented by one of the form
$\alpha:k[t]/(t^p) \longrightarrow \Dist(\hat{U}_1)$ such that $\alpha(t) \in \hat{\FU}$.
Let $x \in \FU$ be an element such that $x+ \langle u_n \rangle = \alpha(t)$. 
We know that as an elements of $\Dist(U_1)$, $x^p = x^{[p]} \in \langle u_n \rangle$
since $\alpha(t)^p = 0$ in $\Dist(\hat{U}_1)$. It follows that either $x^p = 0$ or
$x^p$ is a multiple (by an element of $k$) of $u_n$. In the former case the 
$p$-point $\alpha$ lifts to the $p$-point $\beta:k[t]/(t^p) \longrightarrow \Dist(U_1)$,
given by $\beta(t) = x$.

Therefore, we have shown that $U_1$ satisfies Hypothesis \ref{hyp3}. Hence,
the theorem follows from Theorem \ref{char-endotriv}.
\end{proof}


\section{Endotrivial modules over parabolic subgroups} 

In this section we show that we can use the results of the last 
section to obtain information on infinitesimal subgroups of 
parabolic subgroups. We maintain 
the notation of the last section and the same assumptions on 
the group $G$. First we consider the case of
the Borel subgroup.

For any group scheme $H$, let $\text{mod}(H)$ be the category of
finite dimensional rational $H$-modules. This construction can be applied 
when $H=G$, $B$, $P_{J}$, $L_{J}$, $U$, $U_{J}$ and $T$. Let $X(T_{r})$ be 
the set of characters of $T_{r}$ which can 
be indentified with the set of one dimensional 
simple modules for $T_{r}$. 

\begin{thm} Let 
$U$ be the unipotent radical of $B$. Then we have the following.
\begin{itemize} 
\item[(a)] Suppose that $\dim U=1$. Then $T(U_{1})\cong  \{1\}$ if $p = 2$,
$T(U_{1})\cong {\mathbb Z}/2{\mathbb Z}$ if $p > 2$ and 
$T(U_{r})\cong {\mathbb Z}$ for $r\geq 2$. 
\item[(b)] In any case, $T(B_{1}) \cong X(T_{1})\otimes T(U_{1})$
\end{itemize} 
\end{thm} 

\begin{proof}
To prove part (a) note that the $\Dist(U_r)$ is isomorphic to 
the group algebra of an elementary abelian $p$-subgroup $E$ of rank $r$. 
Consequently, $T(U_r) = T(E)$ and is as given in the theorem. 

For part (b), suppose that $M$ is an indecomposable
endotrivial module for $B_1$.  By the results of the last section, 
the image of the restriction of $T(B_1)$ to $T(U_1)$ 
is generated by the class of $\Omega(k\vert_{U_1})$ and by at most one other class. 
It can be shown, using methods similar to those in \cite{CMN1,CMN2} that 
the restriction map is surjective. In the case that $T(U_1) \cong \bZ$ is 
generated by the class of $\Omega(k)$, this is obvious. For the cases that
$p=2$ and $G$ is of type $A_2$, or $B_2$ we must rely on the fact that
the cohomology $\HHH^*(B_1,k)$ is finitely generated over the subalgebra 
generated by degree 2 elements (cf. \cite[Section 1]{FPa}).
In these siituations we may choose inverse images for the generators of $T(U_1)$ under
the restriction map. One of these can be taken to be $\Omega(k)$.  The other is 
constructed using the cohomological push-out method used in 
\cite{CMN1,CMN2}.

Now recall that an indecomposable 
module over $B_1$ must be indecomposable when restricted to $U_1$.
As a result, if $N$ is a suitable product of the inverse images of the
generators such that $N_{U_1}$ is in the same class as $M_{U_1}$ 
modulo projectives, then $M \otimes N^*$ is in the class of $k$ on 
restriction to $U_1$ and its class contains a one dimensional 
module. In particular, it is
an element of $X(T_1)$ as asserted. 
\end{proof}

For notation in what follows,  let 
$J\subseteq \Delta$, with $U_{J}$ the 
unipotent radical of parabolic subgroup $P_{J}$. 
The Levi subgroup is 
denoted by $L_{J}$ with $P_{J}=L_{J}\ltimes U_{J}$. 

In what follows we assume that $T(U_r) \cong \bZ$. If $r = 1$, then this
is equivalent to the 
assumption that we are not in one of the exceptional cases of Theorem
\ref{unipotent2}.  The exceptional cases can be handled similarly, but with
more work. 

For parabolic subgroups $P_{J}$ where $J$ is a nonempty proper subset of 
$\Delta$, the only endotrivial modules for $(P_{J})_{r}$ arise from 
syzygies of the trivial module. 
 
\begin{thm} Assume that $\dim U\geq 2$ and $T(U_r)$ is generated by $\Omega(k)$. 
Let $J$ be a proper subset of $\Delta$. Then 
$T((P_J)_{r})=X((P_{J})_{r})\otimes {\mathbb Z}$. 
\end{thm} 

\begin{proof} Let $M$ be an indecomposable endotrivial 
module over $(P_J)_{r}$. Then 
$M$ is endotrivial over $U_{r}$. Since $\dim U\geq 2$, it follows that 
$M|_{U_{r}}\cong \Omega^{n}(k)\oplus 
\text{(proj)}$ so we can assume using dimension shifting that 
$M|_{U_{r}}\cong k\oplus \text{(proj)}$. Hence 
$M|_{(U_{J})_{r}}\cong k\oplus \text{(proj)}$. 
Now let $U^{\prime}$ be another unipotent radical 
for some other Borel subgroup $B\subset P_{J}$. 
Then $M|_{U^{\prime}_{r}}\cong \Omega^{t}(k)\oplus \text{(proj)}$ 
for some $t$. However, 
$U_{J}\subseteq U^{\prime}\cap U$ so 
$M|_{(U_J)_{r}}\cong \Omega^{t}(k)\oplus \text{(proj)}$. 
Thus, $t=0$ and $M|_{U^{\prime}_{r}}\cong k\oplus \text{(proj)}$. 

We also note that the condition that $\dim U\geq 2$ insures that the 
rank of $\Phi$ is greater than two, thus $\dim U_{J}\geq 2$ for 
$J$ a proper subset of $\Delta$. Moreover, one can verify that 
$\dim {\mathcal N}({\mathfrak u}_{J})\geq 2$. This insures that 
$(U_J)_{1}$ has nonperiodic cohomology, so that $\Omega^t(k) \not\cong k$ unless $t = 0$.  

Next we observe that $U_{J}$ is contained 
in every maximal unipotent subgroup of $P_J$. Let 
$H$ be a maximal unipotent subgroup and 
take $HU_{J}$. Since $U_{J}$ is a normal 
subgroup of $P_{J}$, $HU_{J}$ is a subalgebra 
and unipotent. Thus by maximality, 
$U_{J}\leq H$. Moreover, one has $H=\hat{U}\ltimes U_{J}$ where $\hat{U}$ is 
a maximal unipotent subgroup in $L_{J}$.  

Let $L=M^{(U_{J})_{r}}$ be the fixed points of $M$ under $(U_{J})_{r}$. We claim that $L$ 
is an endotrivial $(L_{J})_{r}$-module.  Any 
$\pi$-point, $\alpha:k[t]/(t^{p})\rightarrow \text{Dist}((L_{J})_{r})$  factors as 
$\alpha:k[t]/(t^{p})\rightarrow \text{Dist}((U^{\prime})_{r})
\rightarrow \text{Dist}((U^{\prime})_{r}/U_{r}) \rightarrow \text{Dist}((L_{J})_{r})$ 
for some unipotent radical $U^{\prime}$ of a Borel subgroup in $P_{J}$. 
We have $M|_{U^{\prime}_{r}}\cong k\oplus \text{(proj)}$ so 
$$
L|_{U^{\prime}_{r}}=M^{(U_{J})_{r}}|_{U^{\prime}_{r}}\cong 
k\oplus \text{(proj)}^{(U_{J})_{r}}.
$$ 
Note that $\text{(proj)}^{(U_{J})_{r}}$ is a projective 
$U^{\prime}_{r}/U_{r}$-module. 
The restriction of $M$ along $\alpha$ to a  $(k[t]/(t^{p}))$-module decomposes 
as $k\oplus \text{(proj)}$.   So the fixed point module $M^{(U_{J})_{r}}$ 
has constant Jordan type $\text{[1]}+\text{(proj)}$. 
Consequently, $M^{(U_{J})_{r}}$ is 
endotrivial as an $(L_{J})_{r}$-module by Theorem \ref{pi-char-endotriv}.

If $\bar{u}$ be a generator for the socle of $\text{Dist}((U_J)_{r})$, 
then 
$$
\bar{u}M\subseteq M^{(U_J)_{r}}
$$ 
and there is a short exact sequence of $(P_{J})_{r}$-modules 
$$
0\rightarrow \bar{u}M\rightarrow  
M^{(U_J)_{r}}\rightarrow \lambda \rightarrow 0
$$ 
where $\lambda\in X((P_{J})_{r})$. 
By the preceding argument, $\bar{u}M$ 
is a free $\text{Dist}((U^{\prime})_{1}/(U_J)_{r})$, 
so $\bar{u}M$ has constant Jordan type. 
Consequently, $\bar{u}M$ is a projective $(L_{J})_{r}$-module 
because $\bar{u}$ is normalized by elements in $(P_{J})_{r}$. 

It follows that, as an  $(L_{J})_{r}$-module, 
$M^{(U_{J})_{r}}\cong \lambda \oplus \text{(proj)}$. 
Let $\hat{\iota}:\lambda \rightarrow M^{(U_{J})_{r}}$.  We can compose this 
with the inclusion of $M^{(U_{J})_{r}}\subseteq M$ 
to obtain a $(P_{J})_{r}$-homomorphism: 
$\iota:\lambda \rightarrow M$. We obtain a short exact sequence, 
$$0\rightarrow \lambda  \stackrel{\iota}{\rightarrow} M \rightarrow Q \rightarrow 0$$
of $(P_{J})_{r}$-modules. 

Note that $\iota(\lambda)\nsubseteq \bar{u}M$ and 
$M|_{U^{\prime}_{r}}\cong \lambda \oplus R$ where $R$ is projective. So all 
$U^{\prime}_{r}$-fixed points of $R$ are in $\bar{u}M=\bar{u}R$. 
It follows that $\iota(\lambda)\nsubseteq R$. Hence, 
$Q\cong M/\iota(\lambda)$ is a projective $U^{\prime}_{r}$-module because 
$$
Q|_{U^{\prime}_{r}} \ \cong \ (R\oplus k)/\iota(\lambda) \ \cong \ R.
$$ 
Any $\pi$-point factors through some $(U^{\prime})_{r}$ 
thus $Q$ is projective as a module over $(P_{J})_{r}$. Since $M$ 
is indecomposable we must have $M\cong \lambda$.  
\end{proof} 


\section{Restriction of Weyl modules over $SL_2$ to unipotent subschemes}

In this section we adhere to the notation presented in 
Section 5. In addition, let 
the Euclidean space associated with the root system $\Phi$ be 
denoted by ${\mathbb E}$ and
the inner product on ${\mathbb E} $  by $\langle\ , \ \rangle$.
The Weyl group $W$ is the group generated by reflections
associated to the root system $\Phi$ and the affine Weyl group $W_{p}$
is the group generated by $W$ and translations by elements 
in $p{\mathbb Z}\Phi$.
Let $X(T)$ be the integral weight lattice obtained from $\Phi$.
The set $X(T)$ has a partial
ordering defined as follows. If $\lambda,\mu\in X(T)$, then
$\lambda\geq \mu$ if and only if $\lambda - \mu\in \sum_{\alpha\in
\Pi}\mathbb{N}\alpha$.
If $\alpha^{\vee}=2\alpha/\langle\alpha,\alpha\rangle$ is the coroot
corresponding to $\alpha\in \Phi$, then
the set of dominant integral weights is defined by
$$X(T)_{+}=\{\lambda\in X(T):\ 0\leq \langle\lambda,\alpha^{\vee}\rangle\
\text{for all $\alpha \in \Delta$} \}.$$
Furthermore, the set of $p^{r}$-restricted weights is
$$X_{r}(T)=\{\lambda\in X(T):\
0\leq\langle\lambda,\alpha^{\vee}\rangle<p^{r}\,\,\text{for all $\alpha\in
\Delta$}\}.$$
The affine Weyl group $W_{p}$ acts on $X(T)$
via the ``dot action'' given by $w\cdot \lambda= w(\lambda+\rho)-\rho$
where $w\in W_{p}$, $\lambda\in X(T)$, and
$\rho$ is the half sum of positive roots.

For a reductive algebraic group $G$,
the simple modules will be denoted
by $L(\lambda)$ and the induced modules by $H^{0}(\lambda)=\text{
ind}_{B}^{G} \lambda$, where $\lambda\in X(T)_{+}$. 
The Weyl module $V(\lambda)$ has 
simple head $L(\lambda)$ and is defined as 
$V(\lambda)=H^{0}(-w_{0}\lambda)^{*}$. 
For the infinitesimal group scheme $G_{r}$, the
simple modules will be denoted by $L_{r}(\lambda)$. 
If $\lambda\in X_{r}(T)$, then
$L_{r}(\lambda)\cong L(\lambda)$ as modules over $G_r$. Also, 
for $\lambda\in X(T)_{+}$, let $T(\lambda)$ be 
the indecomposable tilting module with highest weight $\lambda$.

For the rest of this section, suppose that  $G = SL_2$. 
Let $U$ be the unipotent radical of $B$ 
(negative Borel subgroup). Note $X(T)_{+}={\mathbb N}$ and $\dim U=1$. 
The algebra $kU_{r}=\text{Dist}(U_{r})$ is a divided power algebra with basis $\{x^i/i!
\ \vert \ i = 0, 1, \dots, p^r-1\}$ with relations 
\[
\frac{x^i}{i!} \cdot \frac{x^j}{j!} = \binom{i+j}{i} \frac{x^{i+j}}{(i+j)!}
\]
for all $i$ and $j$ such that $i+j < p^r$. The product is zero of 
$i+j \geq p^r$.

Let $W$ be the natural 2 dimensional module for $G$. It has a basis
$\{v_1, v_2\}$ such that $xv_1 = v_2$ and $xv_2 = 0$. The Weyl module
$V(\lambda)$ is the $\lambda^{th}$ symmetric product $S^\lambda(W)$ (viewed 
as a quotient of $W^{\otimes \lambda}$). It has dimension
$\lambda+1$.  The highest weight element in $V(\lambda)$ is the element 
$w_0 = v_1 \otimes \dots \otimes v_1$, with $\lambda$ factors in the tensor product. 
The module $V(\lambda)$ has a basis consisting of $w_0$ and all $w_i$, as given
below, for 
$i = 2, \dots,  \lambda$. The element $w_i$ is the sum in $S^{\lambda}(W)$ 
of all monomials of the form 
$v_{s(1)} \otimes v_{s(2)} \otimes \dots \otimes v_{s(\lambda)}$ 
where $s$ runs through the collection of all function  $s:\{1, \dots, \lambda\} 
\to \{1,2\}$ such that $i$ value of $s$ are equal to $2$ and $\lambda-i$
values of $s$ are equal to $1$. Clearly there are $\binom{\lambda}{i}$ 
such monomials. 

Recall that if $L$ and $M$ are $kU$-modules, then the action of $kU$ on 
$M \otimes N$ is given by 
\[
\frac{x^t}{t!} (\ell \otimes m) \ = \ \sum_{j = 0}^t \frac{x^j}{j!} \ell 
\otimes \frac{x^{t-j}}{(t-j)!} m 
\]
for all $\ell \in L$, $m \in M$ and all $t \leq p^r-1$. The aforementioned 
action is imposed by the Hopf algebra structure on $kU$. The first result that we need is the following.

\begin{lemma}
Let $j$ and $t$ be nonnegative integers such that $j+t \leq \lambda$. Then 
\[ 
\frac{x^j}{j!} w_t \ = \ \binom{t+j}{j} w_{t+j}.
\]
If $t+j > \lambda$, then $(x^j/(j!)) w_t = 0$. 
\end{lemma}

\begin{proof}
We notice first that $(x^i/(i!)) v_j = 0$ for $i \geq 2$. As a 
consequence, we have that $(x^i/(i!)) w_0 = w_i$ provided $i <\lambda$
and $i < p^r$. This observation is enough to prove the lemma in the 
case that $t+j < p^r$. To verify the general case we must count 
monomials. Hence, one way to proceed would be to embed $U = U_r$ in 
$U_s$ for $s$ sufficiently large. On the other hand, it is a simple 
task to count monomial terms. 

Suppose that $u = u_1 \otimes \dots \otimes u_{\lambda}$ is a monomial such that 
exactly $t$ of the $u_i's$ are equal to $v_2$ and the rest equal to $v_1$.
Then the expression of $(x^j/(j!)) u$ as a sum of monomial terms,
has $\binom{\lambda-t}{j}$ such terms. Moreover, all such terms have exactly
$t+j$ factors equal to $v_2$ while the rest are equal to $v_1$. That is, 
$x^j/(j!) \cdot w_t$ is a multiple of $w_{t+j}$. Now, $w_{t+j}$
has exactly $\binom{\lambda}{t+j}$ total terms. Because there are 
$\binom{\lambda}{t}$ such $u$ in the expression of $w_t$, we must have
that the coefficient 
on each term in the expression of $(x^j/(j!)) w_t$ is 
\[
\frac{\binom{\lambda-t}{j}\binom{\lambda}{t+j}}{\binom{\lambda}{t}} = 
\frac{(\lambda-t)!}{j! (\lambda-t-j)!} \frac{\lambda!}{(t)! (\lambda-t)!} \frac{(t+j)!(\lambda-t-j)!}{\lambda!}
= \binom{t+j}{j}
\]
as asserted. 
\end{proof}

With this formula we can prove the following result. 

\begin{prop} 
Let $n$ be any positive integer. Then the restriction of $V(np^r)$ 
to $U_{r}$ is the direct sum
\[
V(np^{r})|_{U_{r}}\cong k\oplus (kU_{r})^{\oplus n}.
\]
\end{prop}

\begin{proof}
Recall that $kU_{r}$ is a local ring with a simple socle that is spanned
as a $k$-vector space by the element $x^{p^r-1}/(p^r-1)!$. In particular,
it is a self injective algebra, and hence projective $kU_{r}$-modules are both
free and injective. 

Suppose that $Q$ is free $kU_{r}$-module with $kU_{r}$-generators $q_1, \dots,
q_n$. Define a homomorphism $\psi: Q \longrightarrow V(np^{r})$ by the assignment
$\psi(q_i) = w_{(i-1)(p^r)}$. Thus we have that 
\[
\psi(\frac{x^{p^r-1}}{(p^r-1)!} q_i) = 
\binom{(i-1)p^r + p^r-1}{p^r-1} w_{ip^r-1} \neq 0
\]
The reason why the above expression is not zero is that 
\[
\binom{(i-1)p^r + p^r-1}{p^r-1} = \prod_{j =0}^{p^r-2} 
\frac{ip^r -1-j}{p^r-1-j}
\]
where in each factor the same power of $p$ divides the numerator as 
divides the denominator.  That is, for $0 \leq j \leq p^r-2$, the highest
power of $p$ which divides either $ip^r-1-j$ or $p^r-1-j$ is the same
as the highest power of $p$ which divides $1+j$. 
This proves that the map $\psi$ is injective
since it is injective on the socle of $Q$. Moreover, the image of $\psi$ 
has codimension one in $W$.  Hence we have an exact sequence 
\[
\xymatrix{
0 \ar[r] & Q \ar[r]^{\psi} & W \ar[r] & k \ar[r] & 0
}
\]
which is split because $Q$ is an injective module. 
\end{proof}

\begin{thm} \label{sl2-1}
Let $G=SL_{2}$. The Weyl module $V(\lambda)$ (resp. $H^{0}(\lambda)$) is endotrivial over $G_{r}$ if 
and only if $\lambda=np^{r}$ or $np^{r}-2$ for $n\geq 0$. 
\end{thm}

\begin{proof} First observe that $\dim V(\lambda)=\lambda+1$. If $V(\lambda)|_{G_{r}}$ 
is endotrivial then $\dim V(\lambda)\equiv \pm 1 \ \text{mod }p^{r}$. 
This implies that $\lambda=np^{r}$ or $np^{r}-2$ where $n\geq 0$. 
 
Next we know that $V(np^{r})$ is endotrivial as a $U_{r}$-module. Because $V(np^{r})$ is a 
$G$-module, $V(np^{r})$ is an endotrivial module for any maximal unipotent 
subgroup scheme $U'_{r}$ of $G_{r}$. Now every $\pi$-point factors through 
one of these group schemes. Consequently, the result follows from 
Theorem \ref{pi-char-endotriv}.

Let $T_{\lambda}^{\mu}$ be the translation functor as given in \cite[Chapter 7]{Jan}. 
According to \cite[II. 7.19 Proposition]{Jan}, there exists a short exact sequence 
of $G$-modules 
$$
0\rightarrow V(np^{r}) \rightarrow T_{-1}^{0}(V(np^{r}-1)) \rightarrow V(np^{r}-2) \rightarrow 0.
$$
But, $T_{-1}^{0}(V(np^{r}-1))|_{G_{r}}$ is a projective (also injective) 
$G_{r}$-module, thus 
an injective $U_{r}$-module. It follows that   
$V(np^{r}-2)|_{U_{r}}\cong \Omega^{-1}(k)\oplus (\text{proj})$ and 
$V(np^{r}-2)|_{U_{r}}$ is endotrivial over $U_{r}$. The argument in the preceding 
paragraph can be used to conclude that $V(np^{r}-2)|_{G_{r}}$ is endotrivial. 
\end{proof}

\begin{cor} Let $G=SL_{2}$. 
\begin{itemize} 
\item[(a)] The simple module $L(\lambda)$ is endotrivial over $G_{r}$ if 
and only if $r=1$ and $\lambda=0,p-2$.
\item[(b)] The tilting module $T(\lambda)|_{G_{r}}$ is endotrivial if and only if 
$r=1$ and $\lambda=0,p-2$.
\end{itemize}  
\end{cor}

\begin{proof} (a) Let $\lambda\in X(T)_{+}$ and assume that $L(\lambda)$ is 
endotrivial over $G_{r}$. Then $L(\lambda)|_{G_{1}}$ must be a direct sum of an endotrivial 
$G_{1}$-module with a projective $G_{1}$-module. By Steinberg's Tensor 
Product Theorem, $L(\lambda)=L(\lambda_{0})\otimes L(\mu)^{(1)}$ where 
$\lambda\in X_{1}(T)$ and $\mu\in X(T)_{+}$. Therefore, 
$L(\lambda)|_{G_{1}}\cong L(\lambda_{0})^{\oplus n}$ where $n=\dim 
L(\mu)^{(1)}$. Combining this with the fact mentioned above means that $\mu=0$ 
and $L(\lambda)=L(\lambda_{0})$, thus 
$\lambda\in X_{1}(T)$. 

Now when $\Phi$ is of type $A_{1}$, $V(\lambda)=L(\lambda)$ for all $\lambda\in X_{1}(T)$.
Thus the result follows by using Theorem \ref{sl2-1}. 

(b) Suppose that $T(\lambda)|_{G_{r}}$ is endotrivial. If $\lambda\geq p-1$ then  
$T(\lambda)|_{G_{1}}$ is projective \cite[E.8. 
Lemma]{Jan}, and $p\mid \dim T(\lambda)$. Thus 
we can assume that $0\leq \lambda \leq p-2$. 
The only possibilities for endotrivial 
modules for $T(\lambda)|_{G_{r}}$ occur when $\lambda=0,p-2$ when $r=1$ because 
in these cases $T(\lambda)\cong V(\lambda)$ for $0\leq \lambda \leq p-1$. 
\end{proof} 

Finally, we can determine the structure of the endotrivial group for $G_{1}=(SL_{2})_{1}$. 

\begin{cor} Let $G=SL_{2}$. Then $T(G_{1})\cong {\mathbb Z}
\oplus {\mathbb Z}/2{\mathbb Z}$. 
\end{cor}

\begin{proof} The category of $G_{1}=(SL_{2})_{1}$-modules has tame representation 
type and the indecomposable module are classified (cf. \cite[Section 3]{CNP}). Using this classification, 
one see that 
$$T(G_{1})=\{\Omega^{m}(k),\ \Omega^{n}(L(p-2)):\ m,n\in {\mathbb Z}\}.$$
Moreover, $L(p-2)\cong L(p-2)^{*}$, thus 
$T(G_{1})=\langle\Omega^{1}(k),L(p-2)\rangle\cong {\mathbb Z}\oplus 
{\mathbb Z}/2{\mathbb Z}$. 
\end{proof}


\section{Rational Endotrivial Modules over $G$} 

Our aim is to investigate the question of when a rational $G$-module is 
an endotrivial module over any of the infinitesimal subgroups of $G$.
We show that if the Lie rank of $G$ is at least 2 then the (nontrivial) irreducible 
modules for $G_r$ as well as the (nontrivial) Weyl modules and tilting modules are not 
endotrivial. The situation is strikingly different for $G = SL_2$,
as seen from the previous section.

Let $S$ be an affine algebraic subgroup 
scheme of $G$, and let $M$ be a $S$-module. 
For $g\in G$, one can consider the 
twisted module $M^{g}$ which is a $S^{\prime}=g^{-1}Sg$-module. 
(cf. \cite[I. 2.15]{Jan}) In particular 
if $g$ normalizes $S$ then the twisted module $M^{g}$ becomes a $S$-module. 
Moreover, if $M$ is a $G$-module then $M^{g}\cong M$. 

For $J\subseteq \Delta$, let $L_{J}$ be the Levi subgroup of $G$ determined by $J$. 
Set $X_{J}(T)_{+}=\{\lambda\in X(T):\ 0\leq \langle \lambda,\alpha^{\vee} \rangle \ \text{for all $\alpha\in J$}\}$. 
For $\lambda\in X_{J}(T)_{+}$, one has a nonzero induced module $H^{0}_{J}(\lambda):=\text{ind}^{L_{J}}_{L_{J}\cap B}\lambda$ 
with simple $L_{J}$-socle $L_{J}(\lambda)$. Moreover, one can define the Weyl module $V_{J}(\lambda)$ with head 
$L_{J}(\lambda)$. Let $T_{J}(\lambda)$ be the indecomposable $L_{J}$-tilting module of highest weight $\lambda$. 

Following \cite{Sm} (cf. \cite[II. 5.21]{Jan}) we have a weight space decomposition: 
$$
H^{0}(\lambda)=\left(\bigoplus_{\nu\in 
{\mathbb Z}J}H^{0}(\lambda)_{\lambda-\nu}\right) \oplus M.
$$ 
where $M$ is the direct sum of all weight spaces $H^{0}(\lambda)_{\sigma}$ where 
$\sigma\neq \lambda-\nu$ for any $\nu\in {\mathbb Z}J$. Note that 
$H^{0}_{J}(\lambda)=\oplus_{\nu\in {\mathbb Z}J}H^{0}(\lambda)_{\lambda-\nu}$ and 
the aforementioned decomposition is $L_{J}$-stable. Therefore, 
\begin{equation} 
H^{0}(\lambda)|_{L_{J}}\cong H^{0}_{J}(\lambda)\oplus M. 
\end{equation} 
Dually, we have that 
\begin{equation} 
V(\lambda)|_{L_{J}}\cong V_{J}(\lambda)\oplus M^{\prime}. 
\end{equation} 
for some $L_{J}$-module $M^{\prime}$. 
The same argument works for the simple module $L(\lambda)$ and there is 
a decomposition: 
\begin{equation} 
L(\lambda)|_{L_{J}}\cong L_{J}(\lambda)\oplus N. 
\end{equation} 
for some $L_{J}$-module $N$. 

Observe that $L(\lambda)=\text{soc}_{G}(H^{0}(\lambda))$ which 
implies that $\text{soc}_{L_{J}}L(\lambda)\subseteq \text{soc}_{L_{J}}(H^{0}(\lambda))$. 
Note that 
$L_{J}(\lambda)=\text{soc}_{L_{J}}(H^{0}_{J}(\lambda))\subseteq 
\text{soc}_{L_{J}}(H^{0}(\lambda))$.
Moreover, $L_{J}(\lambda)$ appears as an
$L_{J}$-composition factor of $L(\lambda)$ and 
$H^{0}(\lambda)$ with multiplicity one. 
Therefore, $L_{J}(\lambda)$ must appear  
in $\text{soc}_{L_{J}}L(\lambda)$. By a dual argument, we 
can use the decomposition 
$$V(\lambda)|_{L_{J}}\cong V_{J}(\lambda)\oplus N$$  
to show that $L_{J}(\lambda)$ appears in the 
head of $L(\lambda)|_{L_{J}}$. The fact that 
$L_{J}(\lambda)$ has multiplicity one in $L(\lambda)$ now shows that $L_{J}(\lambda)$ 
is a $L_{J}$-direct summand of $L(\lambda)$.

Finally, we observe that $T(\lambda)|_{L_{J}}$ is a tilting module for the Levi subgroup 
$L_{J}$ and we have that
\begin{equation} 
T(\lambda)|_{L_{J}}\cong T_{J}(\lambda)\oplus T
\end{equation} 
for some $L_{J}$-tilting module $T$.

We first show that the nontrivial induced modules and Weyl 
modules for groups of Lie rank greater than two 
are not endotrivial. 

\begin{thm}  \label{rk2}
Assume that $p>2$. If 
$\operatorname{rank }G=2$ (i.e., $\Phi=A_{2}$, $B_{2}$ 
or $G_{2}$), then the only induced module $H^{0}(\lambda)$ (resp. Weyl module 
$V(\lambda)$) which is endotrivial over $G_{r}$ is the trivial module. 
\end{thm}   

\begin{proof} We first note that if a 
module is endotrivial then the dual of the 
module is also endotrivial so the statement 
of the theorem for induced modules will 
imply the statement for Weyl modules. 

Let us first consider the case when 
$\Phi=A_{2}$ with $p>2$. Let $\lambda = (\lambda_1,\lambda_2)$.  
Observe that 
if $\lambda_{1}$ or $\lambda_{2}$ has 
the form $np^{r}-1$ then by Weyl's Dimension 
Formula, $p\mid \dim H^{0}(\lambda)$ 
thus $H^{0}(\lambda)$ cannot be endotrivial. 
So we can assume that neither $\lambda_{1}$ 
or $\lambda_{2}$ has the form $np^{r}-1$. 
Let $J_{1}=\{\alpha_{1}\}$ and 
$J_{2}=\{\alpha_{2}\}$. Then we can consider the 
following decompositions: 
\begin{equation} 
H^{0}(\lambda)|_{L_{J_{1}}}\cong H^{0}_{J_{1}}(\lambda)\oplus M_{1},
\end{equation} 
\begin{equation} 
H^{0}(\lambda)|_{L_{J_{2}}}\cong H^{0}_{J_{2}}(\lambda)\oplus M_{2}.
\end{equation} 
Set $w_{1}=s_{\alpha_{2}}$ and 
$w_{2}=s_{\alpha_{1}}$. Then $L=w_{1}(L_{J_{1}})=w_{2}(L_{J_{2}})$ where 
$L/[L,L]$ has type $A_{1}$ and contains 
the root subgroup $U_{\alpha_{1}+\alpha_{2}}$. The 
aforementioned decompositions can be 
twisted by $w_{1}$ (resp. $w_{2}$) to obtain 
\begin{equation} 
H^{0}(\lambda)|_{L}\cong H^{0}_{J_{1}}(\lambda)^{w_{1}}\oplus M_{1}^{w_{1}},
\end{equation} 
\begin{equation} 
H^{0}(\lambda)|_{L}\cong H^{0}_{J_{2}}(\lambda)^{w_{2}}\oplus M_{2}^{w_{2}}. 
\end{equation} 
Observe that $H^{0}_{J_{i}}(\lambda)^{w_{i}}$ 
for $i=1,2$ are indecomposable nonisomorphic $L$-module
whose characters are not equal, unless $\lambda_{1}=\lambda_{2}=0$. Thus, 
$H^{0}_{J_{i}}(\lambda)^{w_{i}}$ for $i=1,2$ 
must appear as indecomposable $L$-summands of 
$H^{0}(\lambda)|_{L}$. If $H^{0}(\lambda)|_{G_{r}}$ is 
an endotrivial module, then $H^{0}(\lambda)|_{L_{r}}$ 
is the direct sum of an endotrivial $L_{r}$-module and
a projective $L_{r}$-module. But, $H^{0}(\lambda)|_{L_{r}}$ contains summands 
$H^{0}_{J_{i}}(\lambda)^{w_{i}}|_{L_{r}}$ for $i=1,2$, 
neither of which is projective 
since $\lambda_{i}$ is not of the form $np^{r}-1$ for $i=1,2$. 

For type $B_{2}$, note that $\lambda_{1}$ and $\lambda_{2}$ cannot both 
have the form $np^{r}-1$. For otherwise,  
$p\mid \dim H^{0}(\lambda)$, and hence $H^{0}(\lambda)$ could 
not be endotrivial. 
Suppose that $\lambda_{1}$ is not a multiple 
of $p^{r}-1$. Let $w=s_{\alpha_{1}+\alpha_{2}}$ and 
$J=\{\alpha_{1}\}$.  First note that 
$$
H^{0}(\lambda)|_{L_{J}}\cong H^{0}_{J}(\lambda)\oplus M.
$$
Because $\alpha_{1}+\alpha_{2}$ is orthogonal to 
$\alpha_{1}$, $w(L_{J})=L_{J}$, and 
$H^{0}(\lambda)\cong H^{0}(\lambda)^{w}$ as $G$-modules, 
$$
H^{0}(\lambda)|_{L_{J}}\cong H^{0}_{J}(\lambda)^{w}\oplus M^{w}.
$$
The highest weight of $H^{0}_{J}(\lambda)^{w}$ 
is $w\lambda$. This is given by the 
formula 
\begin{equation} 
s_{\alpha_{1}+\alpha_{2}}\lambda=\lambda-\langle 
\lambda,(\alpha_{1}+\alpha_{2})^{\vee} \rangle 
(\alpha_{1}+\alpha_{2})=(\lambda_{1},-\lambda_{1}-\lambda_{2}). 
\end{equation}
This demonstrates that the indecomposable $L_{J}$-modules 
$H^{0}_{J}(\lambda)$ and $H^{0}_{J}(\lambda)^{w}$ both 
appear as summands of $H^{0}(\lambda)|_{L_{J}}$. 
But, if $H^{0}(\lambda)|_{G_{r}}$ is 
endotrivial then $H^{0}(\lambda)_{L_{J}}$ can have 
at most one endotrivial direct summand 
which can only occur if $\lambda_{1}=\lambda_{2}=0$. 
A similar argument can be used for the 
case when $\lambda_{2}$ is not a multiple of $p^{r}-1$. 
In this case, one can take $J=\{\alpha_{2}\}$ with 
$w=s_{2\alpha_{1}+\alpha_{2}}$, so that  
\begin{equation} 
s_{2\alpha_{1}+\alpha_{2}}\lambda=(-\lambda_{1}-2\lambda_{2},\lambda_{2}). 
\end{equation}

The argument for $\Phi=G_{2}$ follows 
the same line of reasoning as in the $B_{2}$-case with 
the existence of orthogonal roots. 
When $J=\{\alpha_{1}\}$ (resp. $J=\{\alpha_{2}\}$), 
take $w=s_{3\alpha_{1}+2\alpha_{2}}$ 
(resp. $w=s_{2\alpha_{1}+\alpha_{2}}$) and observe that 
\begin{equation} 
s_{3\alpha_{1}+2\alpha_{2}}\lambda=(\lambda_{1},-\lambda_{1}-\lambda_{2}),  
\end{equation}
\begin{equation} 
s_{2\alpha_{1}+\alpha_{2}}\lambda=(-\lambda_{1}-3\lambda_{2},\lambda_{2}). 
\end{equation}
\end{proof} 

\begin{thm} \label{weyl7}
Suppose that $p>2$ and that $\operatorname{rank }G\geq 2$. 
Then the only induced module $H^{0}(\lambda)$ (resp. Weyl module 
$V(\lambda)$) which is endotrivial over $G_{r}$ is the trivial module. 
\end{thm}   

\begin{proof} It suffices to prove that statement for the induced modules 
$H^{0}(\lambda)$ with $\lambda=\sum_{i=1}^{n}\lambda_{i}\omega_{i}$. 
Choose adjacent simple roots $\alpha_{i}$, $\alpha_{j}$ in the Dynkin 
diagram and set $J=\{\alpha_{i},\alpha_{j}\}\subseteq \Delta$. 
Now 
$$H^{0}(\lambda)|_{L_{J}}\cong H^{0}_{J}(\lambda)\oplus M.$$ 
Therefore, $H^{0}_{J}(\lambda)$ must be endotrivial over $(L_{J})_{r}$. 
From the preceding theorem, we have $\lambda_{i}=\lambda_{j}=0$. Since 
this occurs for all pairs of adjacent simple roots, we conclude that $\lambda=0$. 
\end{proof} 

The next result demonstrates that the restrictions of simple nontrivial $G$-modules 
for semisimple groups of rank larger that 
1 cannot be endotrivial modules for $G_{r}$. 
Note that there is no restriction on the prime in this result. 

\begin{thm} Suppose that $\operatorname{rank }G\geq 2$. 
Then the only irreducible $G$-module $L(\lambda)$ 
which is endotrivial over $G_{r}$ is the trivial module. 
\end{thm}   

\begin{proof} Recall first from the proof of Corollary 6.4(a), that we can assume that 
$\lambda\in X_{1}(T)$. 
Now let $J=\{\alpha_{i}\}$ where $\alpha_{i}\in \Delta$. Then 
$$
L(\lambda)|_{(L_{J})_{1}}\cong L_{J}(\lambda)\oplus M.
$$
where $L_{J}(\lambda)$ must either be endotrivial or projective. Our 
analysis of the $A_{1}$ case implies that $\lambda_{i}=0$, $p-2$ or $p-1$. 

One can now use the rank 2 argument involving twisting the decomposition 
$$
L(\lambda)|_{L_{J}}\cong L_{J}(\lambda)\oplus M
$$ 
via Weyl group elements given for induced modules and Weyl modules in 
the proof of Theorem \ref{rk2} to 
show that the endotriviality of $L(\lambda)|_{G_{r}}$ 
implies that $\lambda=0$ 
for $p\geq 3$. The remaining $p=2$ case for $\Phi=A_{2}$ (resp. $B_{2}$, $G_{2}$) can be handled 
as follows. We can enumerate the different 
possibilities for $\lambda\in X_{1}(T)$. The module $L(1,0)$, $L(0,1)$, $L(1,1)$ have dimensions 
which are never congruent to $\pm 1$, modulo $2^{3r}$ (resp. $2^{4r}$, $2^{6r}$), 
thus cannot be endotrivial 
over $G_{r}$. 

Finally, one can use the argument presented in 
the proof of Theorem \ref{weyl7} to handle 
the general case by restricting to irreducible 
rank $2$ subroot systems of $\Phi$ to 
show that $\lambda=0$. 
\end{proof} 

We now show that endotrivial modules rarely 
occur as restrictions of tilting modules. 
Note that if a tilting module $T(\lambda)$, which satisfies the condition that 
$-w_{0}\lambda=\lambda$, were endotrivial, 
then this module would generate a subgroup of order 2 
inside of $T(G_{r})$ because $T(\lambda)\cong T(\lambda)^{*}$. 

\begin{thm} \label{tilting7}
Let $\lambda\in X(T)_{+}$ and let $T(\lambda)$ be the corresponding 
indecomposable tilting module. Suppose that $\operatorname{rank }G\geq 2$ with $p\geq h$. 
Then the only indecomposable tilting module $T(\lambda)$ 
which is endotrivial over $G_{r}$ is the trivial module. 
\end{thm}   

\begin{proof} For the rank two cases, one can now use 
the argument given in Theorem \ref{rk2} and the 
rank one analysis above, to prove that if $T(\lambda)|_{G_{r}}$ is endotrivial 
for $\Phi=A_{2}$, $B_{2}$ or $G_{2}$, then $\lambda=0$. 

For the general case, we first note that if $p\geq h$, then 
the assumption that $T(\lambda)|_{G_{r}}$ is an endotrivial module implies 
that $0\leq \lambda_{j} \leq p-2$ for 
$j=1,2,\dots,n$. If $\lambda_{j}>p-1$ then 
$p\mid T(\lambda)$ by \cite[Theorem 5.3]{CLNP} 
and $T(\lambda)|_{G_{r}}$ is not endotrivial. 
Choose adjacent simple roots $\alpha_{i}$, $\alpha_{j}$ in the Dynkin 
diagram and set $J=\{\alpha_{i},\alpha_{j}\}\subseteq \Delta$. 
Now 
$$T(\lambda)|_{L_{J}}\cong T_{J}(\lambda)\oplus M.$$ 
Since $0\geq \lambda_{i},\lambda_{j} <p-2$, 
$T_{J}(\lambda)$ is not projective (cf. 
\cite[E.8. Lemma]{Jan}), thus must be 
endotrivial so $\lambda_{i}=\lambda_{j}=0$. As this works for any pair of adjacent
roots, we have that $\lambda=0$. 
\end{proof}


\section{Connection with Picard Groups}

In this section we indicate how our work fits in with the triangular 
geometry introduced by Balmer. Given a finite dimensional cocommutative 
Hopf algebra $A$, we denote the stable module category by
${\mathcal K}=A$-stab. This is the quotient of 
the category of $A$-modules 
by maps that factor through projective modules. 
The category ${\mathcal K}$ is a tensor 
triangulated category and thus inherits a ``tensor triangulated geometry'' as 
introduced by Balmer \cite{Bal1}. 

Let $\text{Spc}({\mathcal K})$ be the spectrum of ${\mathcal K}$. As pointed out in \cite{Bal2}, 
$\text{Spc}({\mathcal K})$ has a sheaf of commutative rings 
${\mathcal O}_{\mathcal K}$ which makes it a 
locally ringed spaced. This locally 
ringed space is denoted by $\text{Spec}({\mathcal K})$.   

In \cite[Theorem 6.3(b)]{Bal1}, it is shown that 
$$\text{Spec}(\text{$A$-stab})\cong \text{Proj}({\mathcal V}_{A}(k))
\cong \text{Proj}(\text{H}^{\bullet}(A,k)).$$ 
By interpreting Balmer's construction \cite[Construction 2.6]{Bal2} in 
our setting we have group homomorphisms. 
$$\text{Pic}(\text{Proj}({\mathcal V}_{A}(k))\stackrel{\alpha}{\leftarrow} 
\text{Pic}_{loc tr.}({\mathcal K})\hookrightarrow T(A).$$
where $\text{Pic}$ denote the group of line bundles and 
$T(A)$ is the group of endotrivial modules. Moreover, 
we have the following theorem. 

\begin{thm} Let $A$ be a finite dimensional cocommutative Hopf algebra. Then 
there exists a monomorphism of groups 
$$\beta:\operatorname{Pic}(\operatorname{Proj}({\mathcal V}_{A}(k))\otimes_{\mathbb Z}{\mathbb Z}[1/p]
\hookrightarrow T(A)\otimes_{\mathbb Z}{\mathbb Z}[1/p].$$ 
\end{thm} 

We remark that Balmer proves using the gluing techniques of Balmer-Favi \cite{BF} and
Balmer-Benson-Carlson \cite{BBC} that 
$\beta$ becomes an isomorphism when $A=kG$ where $G$ is a finite group 
after tensoring by ${\mathbb Q}$ (cf. \cite[Theorem 4.7]{Bal2}). 
It is still open if 
this holds for arbitrary finite group schemes. 

Let us look at two examples in the setting of infinitesimal group schemes. 

\begin{exm} Let $A=\text{Dist}(G_{1})$ where $G=SL_{2}$. We have seen that $T(A)={\mathbb Z}\oplus {\mathbb Z}/2{\mathbb Z}$. 
On the other hand, if $p\geq 3$, then ${\mathcal V}_{A}(k)\cong {\mathcal N}$ where ${\mathcal N}$ 
is the set of $2\times 2$ nilpotent matrices. In this case, $\text{Pic}(\text{Proj}({\mathcal N}))\cong {\mathbb Z}$. 
This shows that the Picard group and the endotrivial groups do not identify themselves on the integral level. 
\end{exm} 

\begin{exm} Let $A=\text{Dist}(U_{1})$ where $U$ is the unipotent radical of a Borel subgroup. 
If $p\neq 2$ when the underlying root system is of type $A_{2}$ or $B_{2}$ then 
$T(U_{1})\cong {\mathbb Z}$. We know that 
$\operatorname{Pic}(\operatorname{Proj}({\mathcal V}_{U_{1}}(k)))$ 
has rank at least one. The theorem above demonstrates that $\text{rank } \operatorname{Pic}(\operatorname{Proj}
({\mathcal V}_{U_{1}}(k)))=1$. 
\end{exm}   


\section{Open Problems:} 

We conclude this paper by presenting several open questions which the authors 
view as worthy of further investigation. 

Let $P$ be a projective indecomposable $G_{r}$-module. It has been a long 
standing conjecture that the $G_{r}$-action of 
$P$ should lift to a rational action 
of $G$. Ballard showed that this holds when 
$p\geq 3(h-1)$ and Jantzen proved this 
when $p\geq 2(h-1)$. For a discussion of this 
problem, see \cite[II. 11.11]{Jan}. 
This motivates our first question. 
\vskip .25cm 
\noindent
(1) Let $G$ be a reductive algebraic group and let $M$ be an endotrivial 
$G_{r}$-module. Does $M$ lift to a $G$-structure? 
\vskip .25cm 
We have seen that in most cases our computations 
show that $T(H_{r})\cong {\mathbb Z}$ 
where $H$ is an affine algebraic group scheme. 
This provides some basis for asking the following 
two questions: 
\vskip .25cm 
\noindent
(2) Let $H$ be an arbitrary finite group scheme. Is $T(H)$ finitely generated? 
For finite groups this was first proved by 
Puig \cite{Pu}, but also follows from the 
the classification of endotrivial modules over $p$-groups \cite{CaTh3}. 
\vskip .25cm 
The final two questions entail a finding a more 
detailed description of our computations 
in Sections 5-8. 
\vskip .25cm 
\noindent
(3) Suppose that $U$ is the unipotent radical of the Borel subgroup of a 
semisimple algebraic group $G$. For $r \geq 2$, does $A=\text{Dist}(U_{r})$ 
satisfy Hypothesis~\ref{hyp3}, or more 
generally is $T(U_r)$ generated by $\Omega(k)$?
\vskip .25cm 
\noindent
(4) Determine the structure of $T(({SL}_{2})_{r})$ 
and in general $T(G_{r})$ where $G$ is 
a reductive algebraic group scheme. This is the 
first step in solving the more general 
question of finding $T(G)$ for $G$ a finite group scheme.

\vskip.3in

\end{document}